\documentclass[11pt]{amsart}
\usepackage{amsmath}
\usepackage{amsthm}
\usepackage{amsfonts}
\usepackage{amssymb}
\usepackage{verbatim}
\usepackage{graphicx}
\usepackage[margin=1.0in]{geometry}
\usepackage{url}
\usepackage{enumerate}
\usepackage[pdftex,bookmarks=true]{hyperref}

\usepackage{color}

\usepackage[normalem]{ulem}
\newcommand\ep{\varepsilon}

\newcommand\R{{\mathbb{R}}}

\newcommand\Z{{\mathbf{Z}}}

\renewcommand\P{{\mathbf{P}}}
\newcommand\E{{\mathbf{E}}}

\newcommand\Bx{{\mathbf x}}

\newcommand\CB{{\mathcal B}}

\newcommand\CE{{\mathcal E}}

\theoremstyle{plain}
  \newtheorem{theorem}[subsection]{Theorem}

  \newtheorem{fact}[subsection]{Fact}
  \newtheorem{lemma}[subsection]{Lemma}
  \newtheorem{corollary}[subsection]{Corollary}

  \newtheorem{remark}[subsection]{Remark}
  
  \newtheorem{claim}[subsection]{Claim}

\theoremstyle{definition}
  \newtheorem{definition}[subsection]{Definition}

\include{psfig}

\include{psfig}

\author{Yen Do}
\email{yen.do@virginia.edu}
\address{Department of Mathematics, The University of Virginia, Charlottesville, VA 22904}

\author{Hoi Nguyen}
\email{nguyen.1261@math.osu.edu}
\address{Department of Mathematics, The Ohio State University, Columbus, OH 43210}

\author{Van Vu}
\email{van.vu@yale.edu}
\address{Department of Mathematics, Yale University, New Haven, CT 06511}

\thanks{Y. Do is supported in part by research grant DMS-1201456.}
\thanks{H. Nguyen is supported by research grant DMS-1358648}
\thanks{V. Vu is supported by research grants DMS-0901216 and AFOSAR-FA-9550-12-1-0083}

\begin{document}

\title{Real roots of random polynomials: expectation and repulsion}

\begin{abstract} Let $P_{n}(x)= \sum_{i=0}^n \xi_i x^i$ be a Kac random polynomial where the coefficients $\xi_i$ are iid copies of a given random variable $\xi$.  Our main result   is an optimal quantitative bound concerning real roots repulsion. This leads to an optimal  bound on the probability that there is a double root. 

As an application, we consider the problem of estimating  the number of real roots of $P_n$, which has a long history and in particular was the main subject of a celebrated series of papers by 
Littlewood and Offord from the 1940s.  We show, for a large and natural family of atom variables $\xi$,  that  the expected number of real roots of $P_n(x)$ 
is exactly  $\frac{2}{\pi} \log n +C +o(1)$, where $C$ is an absolute constant depending on  the atom variable $\xi$.  Prior to this paper, such a result was 
known only for the case when $\xi$ is Gaussian. 
\end{abstract}

\maketitle

\section{Introduction} \label{section:intro}

%{\color{blue}To Hoi: I changed some words belows for clarity.}

Let $\xi$ be a real random variable having no atom at   0, zero mean and unit variance. Our object of study is  the random  polynomial 
\begin{equation} \label{defP} P_{n}(x) := \sum_{i=0}^n \xi_i x^i \end{equation}  where $\xi_i$ are iid copies of $\xi$. This polynomial is often referred to as Kac's polynomial, and has been extensively investigated 
in the literature.

\subsection{Real roots of random polynomials} The study of real roots of random polynomials has a long history.  Let  $N_n$ be the number of real roots of $P_n(x)$, sometimes we use the notation $N_{n,\xi}$ to emphasize the dependence of $N_n$ on the distribution of $\xi$.  This is a random variable taking values in $\{0,\ldots,n\}$.  The issue of estimating  $N_n$ was already raised by Waring as far back as 1782 (\cite[page 618]{To}, \cite{Kostlan}), and has generated a large amount of literature, of which we  give an (incomplete and brief) survey. 

  One of the first estimates for $N_n$ was obtained  by Bloch and P\'olya \cite{BP}, who  studied the case when $\xi$ is uniformly distributed in $\{-1,0,1\}$, and established the  upper bound
$$ \E N_n =O( n^{1/2}). $$ Here and later we'll  use the usual asymptotic notation $X=O(Y)$ or $X \ll Y$ to denote the bound $|X| \leq CY$ where $C$ is independent of $Y$. 
 The above bound of Bloch and P\'olya was  not sharp, as it turned out later that $P_n$ has  a remarkably small number of real roots.  In  a  series of  breakthrough papers \cite{lo,  lo-3, lo-4, lo-2} in the early 1940s, Littlewood and Offord proved 
 (for  many  atom variables  $\xi$  such as Gaussian, Bernoulli or uniform on $[-1,1]$) that 
$$ \frac{\log n}{\log \log \log n} \ll N_n \ll \log^2 n $$
with probability $1-o(1)$, where we use $o(1)$ to denote a quantity that goes to $0$ as $n \to \infty$.
  
Around the same time, Kac \cite{kac-0} developed a  general formula for the expectation of number of real roots
\begin{equation} \label{Kacformula}     \E N_n = \int_{-\infty} ^ {\infty} dt \int_{- \infty} ^{\infty} |y| p(t,0,y) dy,
\end{equation}  where 
$p(t,x,y)$ is the probability density for $P_n (t) =x$ ad $P'_n (t) =y$.

In the Gaussian case, one can easily evaluate the RHS and get 
\begin{equation}  \label{Kacformula2}   \E N_n = \frac{1}{\pi} \int_{-\infty} ^{\infty} \sqrt { \frac{1}{(t^2-1) ^2} + \frac{(n+1)^2 t^{2n}}{ (t^{2n+2} -1)^2}  }  dt       =   (\frac{2}{\pi} +o(1)) \log n. \end{equation}

For non-Gaussian distributions, however, Kac's formula is often very hard to evaluate and  it took a considerable amount of work to extend  \eqref{Kacformula2} to  other distributions.  In a subsequent paper \cite{kac-1}, Kac himself handled the case when  $\xi$ is uniformly distributed  on the interval  $[-1,1]$ and  Stevens \cite{Stev}  extended it further  to cover  a large class  of  $\xi$ having continuous and smooth  distributions
with certain regularity properties (see \cite[page 457]{Stev} for details). These papers
relies on \eqref{Kacformula} and  the analytic properties of the distribution of $\xi$.

For discrete distributions, \eqref{Kacformula} does not appear useful and it took more than 10 years since Kac's paper  until Erd\H{o}s and Offord  in 1956 \cite{EO} found a 
completely new approach  to handle the Bernoulli case.  For this case, they  proved that with probability $1- o(\frac{1} {\sqrt {\log \log n} } )$
\begin{equation} 
N_{n, \xi}  = \frac{2}{\pi} \log n + o(\log^{2/3} n \log \log n). 
\end{equation}

In  the  late 1960s and early 1970s,  Ibragimov and Maslova \cite{IM1, IM2}  successfully refined  Erd\H{o}s-Offord's method to handle 
any variable $\xi$ with mean 0.  They proved  that for any $\xi$ with mean zero which belong to the domain of attraction of the normal law, 
\begin{equation}  \label{IM-1} 
\E  N_{n, \xi}  = \frac{2}{\pi} \log n + o(\log n) .
\end{equation} 

 For related results, see also \cite{IM3, IM4}. 
Few years later, Maslova  \cite{Mas1, Mas2} showed that if $\xi$ has mean zero and variance one and $\P (\xi =0) =0$, then the variance of $N_{n, \xi}$ is 
$(\frac{4}{\pi} (1- \frac{2}{\pi} )+o(1) ) \log n$.

Other developments were made in the late 1980s by Wilkins \cite{Wil} and in the early 1990s by Edelman and Kostlan \cite{EK}, who evaluated  the explicit integral in \eqref{Kacformula2} very carefully
and provided a precise estimate for  $\E N_{n, N(0,1)}$ 
\begin{equation} \label{eqn:EK} 
  \E N_{n, N(0,1) }  =  \frac{2}{\pi} \log n  + C_{Gau}  +o(1). 
\end{equation}

where $C_{Gau} \approx .625738072..$ is an explicit  constant (the value of an explicit, but complicated integral).  
As  a matter of fact, one  can even write $o(1)$ as sum of explicit functions of $n$, which gives a complete Taylor expansion.

%{\color{blue}{\bf To Hoi:} It seems that Edelman and Kostlan were not the first people who obtained this expansion.   For details, google ``Edelman Kostlan erratum". }
%{\color{blue} To Yen: yes, incidentally, Van also mentioned this after submitting our joint paper with Oanh. }

The   truly  remarkable fact about  \eqref{eqn:EK}  is that 
the error term $\E N_{n, N(0,1)} - \frac{2}{ \pi} \log n $ tends to a limit as $n$ tends to infinity. The question here is:   Is this  a {\it universal phenomenon}, which holds for 
general random polynomials, or a special property of the Gaussian one ?

It is clear that the computation leading  to \eqref{eqn:EK} is not applicable for general random polynomials, as 
 the explicit formula  in \eqref{Kacformula2}  is available only in the Gaussian case, thanks 
 to the unitary invariance property of this particular distribution. For many natural variables, such as Bernoulli, there is  little hope that 
 such an explicit formula actually exists.   As a matter of fact, Ibragimov-Maslova's proof of  their  asymptotic result  for general non-Gaussian polynomials 
 (based on the earlier work of Erd\H{o}s-Offord) is a tour-de-force. Among others, they followed the Erd\H{o}s-Offord's idea of using  the number of sign changes on a fixed sequence of points to approximate the number of 
 roots. The  error term in this approximation, by nature, has a be large (at least a positive power of $\log n$).

 On the other hand, numerical evidence tends  to support 
the  conjecture that $\E N_n - \frac{2}{\pi} \log n$ do go  to a limit, as $n$ to tends to infinity. However,  the situation is delicate as this limit seems to  depend on the distribution of 
the atom variable $\xi$ and {\it is not}  universal; see the numerical illustration below. 
%{\color{red}\bf Include a few pictures.} 

\begin{figure}[h]
    \centering
    \includegraphics{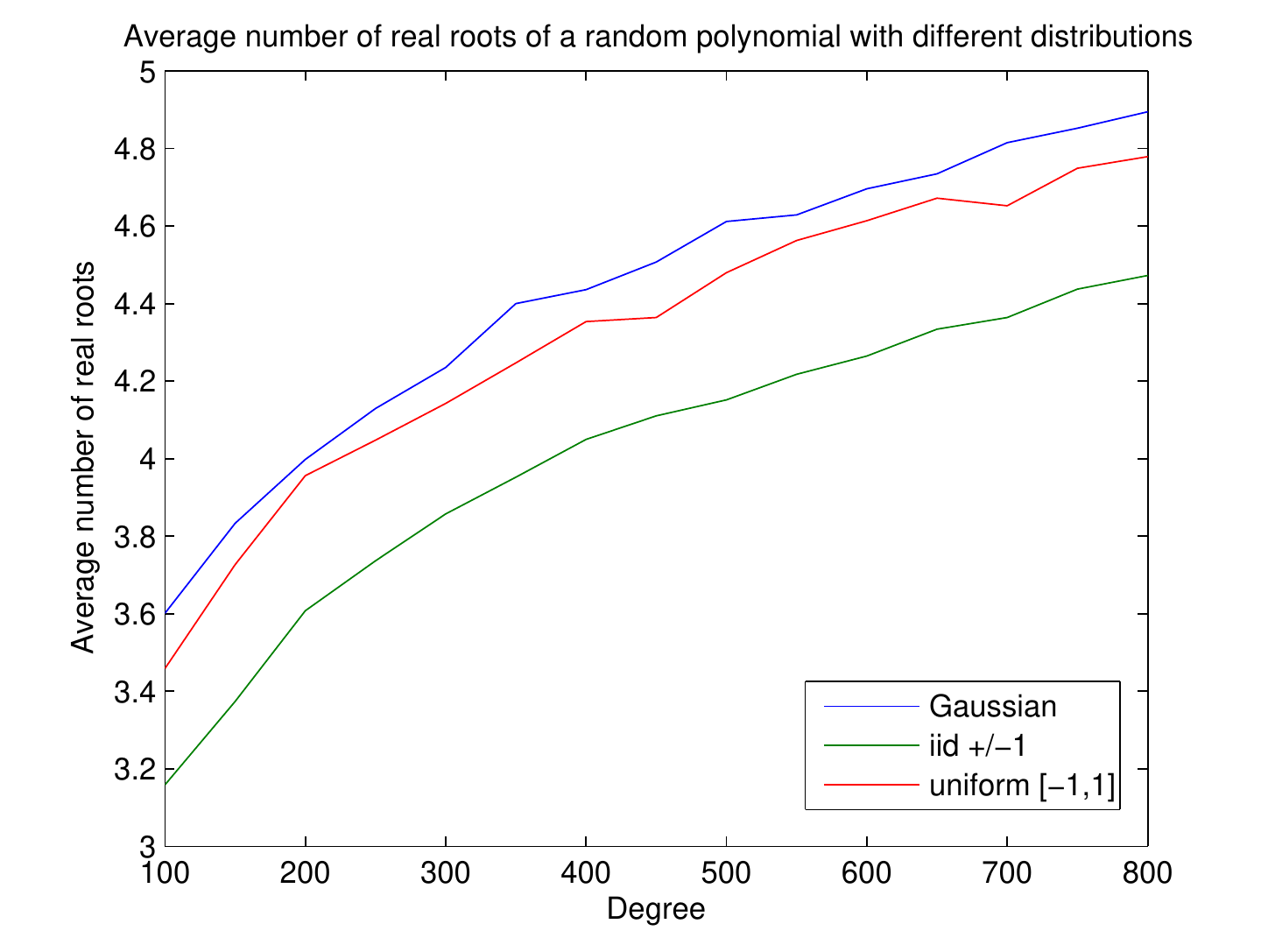}
    \caption{The figure  shows average numbers of real roots of random polynomials with different distributions. One can see the shape of 
the curve $\frac{2}{\pi} \log n$ in all three cases.}
    \label{fig:1}
\end{figure}

%\begin{center} 
%\scalebox{.9}{\includegraphics{figure1.pdf}}
%\end{center}

\begin{figure}[h]
    \centering
    \includegraphics{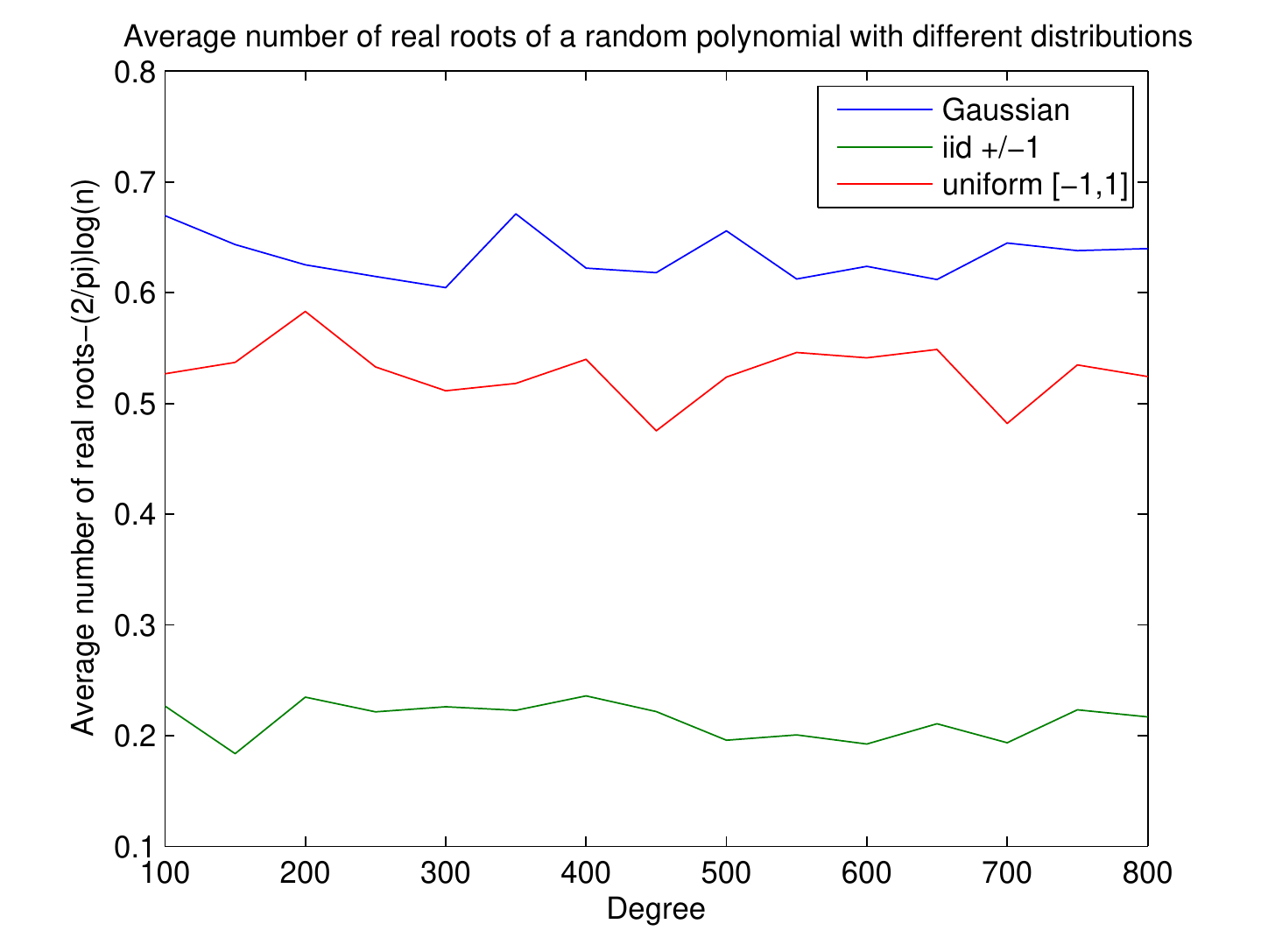}
    \caption{In this second figure, we subtract $\frac{2}{\pi} \log n$ from the averages, and the curves seem to converge to different values.}
    \label{fig:2}
\end{figure}

%\begin{center} 
%\scalebox{.9}{\includegraphics{figure2.pdf}}
%\end{center}

%{\it The figure  shows average numbers of real roots of random polynomials with different distributions, and one can see the shape of 
%the curve $\frac{2}{\pi} \log n$ in all three cases. In the second figure, we subtract $\frac{2}{\pi} \log n$ from this average, and the curves seem to converge to different values. } 

In a recent work  \cite{NgOV}, the last two authors and Oanh Nguyen made a first step by showing  that the error term in question is bounded.

\begin{theorem}  Let $\xi$ be a random variable with mean 0 and variance 1 and bounded $(2 +\epsilon)$-moment. Then 
\begin{equation}\label{eqn:NgOV} 
| \E N_{n, \xi}  - \frac{2}{\pi} \log n |  = O_{\epsilon, \xi} (1). 
\end{equation}

\end{theorem}

The approach in \cite{NgOV}, however,  reveals little about the behavior  of the bounded quantity $O_{\epsilon, \xi} (1)$. 

In this paper, we settle this problem in the affirmative for a large and natural class of distributions, as an application of a general theorem concerning the repulsion between  real roots of 
Kac's polynomials (see the next section).

%\begin{theorem}[Application, Bernoulli case]\label{theorem:expectedvalue:1} Assume that $\xi_i$ are iid Bernoulli random variables  (taking vaues $\pm 1$ independently with probability 1/2),  then
%\begin{equation} \label{eqn:EK-Ber} 
%  \E N_{n,\xi}  =  \frac{2}{\pi} \log n  + C_{Ber} + o(1),
%\end{equation} 
%for some absolute constant $C_{Ber}$.
%\end{theorem}

%As a matter of fact, our method can be generalized to  all discrete uniform distributions without atom at zero.  

\begin{definition}[Type I, discrete distributions]\label{d.uniform-N}
For any positive integer $N$, we say that $\xi$ has  uniform distribution with parameter $N$ (or {\it type I}) if $\P(\xi=i)=1/(2N)$ independently, $i\in \{\pm 1, \pm 2, \dots, \pm N\}$. In particularly, Bernoulli random variable has uniform distribution with parameter 1.
\end{definition} 

\begin{definition}[Type II, continuous distributions]\label{d.p} Let $\ep_0>0$ and $p>1$. We say that a random variable $\xi$ of mean zero has {\it type II} distribution with parameter $(p,\ep_0)$ if its has a $p$-integrable density function and its
$(2+\ep_0)$-moment is bounded.
\end{definition}

\begin{theorem}[Main application: Expectation of real roots] \label{theorem:expectedvalue:2} 
Let $\xi$ be a random variable with either type I or type II with fixed parameters. Then
$$\E N_{n,\xi}  =  \frac{2}{\pi} \log n  + C + o(1),$$
where $C$ is an absolute constant depending on $\xi$.  
\end{theorem}

We would like to mention that due to the abstract nature of our arguments, the exact value of $C$ is still unknown, and its determination remains a tantalizing question.  In our favorite toy case when $\xi$ is Bernoulli, computer simulation suggests that  $C$ is  around  $.22$, but  it looks already difficult to prove  that $|C| \le 10$ (say).

%{\color{blue} To Hoi: I am nervous about claiming the main theorem for non-identical distributions here, so I removed the remark (sorry!). The results about  repulsion of real roots (i.e. Theorem~\ref{theorem:repulsion:uniform} etc.) are certainly true for non-iid settings (and we need them to avoid the symmetric requirement on the distribution). However when applying the repulsion result to prove Theorem~\ref{theorem:expectedvalue:2}, the roles of $P_n(x)$ and $x^n P_n(1/x)$ are not really symmetric when sending $n\to\infty$ (for instance you can not reuse the comparision argument for $|P_n(x) - P_m(x)|$ to estimate $|x^nP_n(1/x) - x^m P_m(1/x)|$).}

%{\color{blue} To anh Van: we added the above paragraph, as we could handle the non-identical distribution case for type II in the proof of no double root.}

%\end{document} 

%{\color{blue} To Yen: I restated the definition and theorem here.}

%{\color{blue} To Yen: Perhaps we should play around with some simulations?}

Now we are going to present the  main  technical contribution of our paper, which, together with some other tools, will yield  Theorem \ref{theorem:expectedvalue:2}  as an application.
The object of study here is  the existence of double roots, or more generally, the repulsion between real roots of Kac polynomials.

\subsection{Roots repulsion of random polynomials}  Multiple roots is a basic  object in theoretical analysis. They 
also play an important role in practical computational problems.  For example, it is  a fundamental fact in numerical analysis that 
 the running time of  Newton's method (for finding real roots of a function) decreases exponentially  with the presence of a multiple or near multiple root (i.e., a place $x$ where both $|P(x)|$ and $|P'(x)|$ are close to zero), see for instance \cite[Chapter 8]{BCSS}. 
 
 Intuitively, one expects that a random polynomial does not have multiple or near multiple roots, with high probability. 
Motivated by problems in complexity theory,  researchers have  confirmed this intuition for the case when  $\xi$ is Gaussian, in a more general setting (see   \cite[Chapter 13]{BCSS} and \cite{CKMW}). Unfortunately, the methods in these works rely on the invariance property of the gaussian distribution and do not extend to other natural distributions. 

We are going to introduce  a new approach that enables us to fully understand the 
double root phenomenon, and consequently 
 derive an optimal bound on the probability that a random polynomial has double or near double real  roots.

 For the sake of presentation,  let us first consider the toy case when $\xi$ is Bernoulli. One expects that the probability of having  double roots tends to 0 with $n$; the question is how fast ?   
To give a lower bound, let us  consider $p_{\pm 1}$, the probability that the polynomial has a double root at either $1$ or $-1$ (we stress that $p_{\pm 1}$ denotes one number).  Trivially, the double root probability is at least 
$p_{\pm 1}$. A short consideration shows (see Appendix \ref{appendix:1})

\begin{fact}\label{fact:arithmetic} 
Assume that $\xi_j$'s have the Bernoulli distribution. Then $p_{\pm 1} = \Theta (n^{-2} )$ if $4|(n+1)$ and $0$ otherwise.
\end{fact} 
 
 % {\bf Hoi: can you double check the odd case ?} 

We are now ready to state our first theorem, which asserts that in the $4|(n+1)$ case, the double root probability is dominating by $p_{\pm 1}$, while in the other cases, this probability is very small. 
 
 \begin{theorem} [Double roots of Bernoully polynomials] \label{theorem:dbroot:Ber} Let $P_n(x)$ be a Bernoulli polynomial. Then 
 $$\P (P_n \,\, \hbox {\rm has real double roots}) = p_{\pm 1} +  n^{-\omega (1)} . $$ 
 \end{theorem}

%Notice that if $n$ is even, $p_{\pm 1}$ is 0. Thus, in this case, the double roots probability is only $n^{-\omega 1} $.

%\end{document} 

%{\color{blue} To Yen: I erased the  repulsion part for Bernoulli.}

%\begin{theorem}[Main result, Bernoulli case]\label{theorem:repulsion:Ber}Assume that $\xi$ is Bernoulli random variable. Then there exists a constant $B>0 $ such that for all sufficiently large $n$ 
%$$\P\Big( \exists x\in \R: |P_n(x)|\le n^{-B} \wedge |P_n'(x)| \le n^{-B}\Big) =O(n^{-2}).$$
%\end{theorem}

%Even more general, we will justify this result for all random variables with uniform distributions. 

Theorem \ref{theorem:dbroot:Ber} is a special case of our main result below, which deals with near double roots of general random polynomials.

\begin{theorem} [Main result: Roots repulsion] \label{theorem:repulsion:uniform} Assume that $\xi_0,\dots, \xi_n$ are independent and all of them are either Type I with the same fixed parameter $N$, or  Type II with uniform implicit constants (which are fixed).  Then for any constant $C>0$ 
 there exists a constant $B>0 $ such that for all sufficiently large $n$
$$\P \Big(\exists x\in  \mathbb R: P_n(x) = 0 \wedge |P_n'(x)|   \le n^{-B} \Big) \le  p_{\pm 1} + n^{-C} $$
where  $p_{\pm 1}$ denotes the probability that the polynomial has a double root at either $1$ or $-1$.
\end{theorem}

%{\color{blue} 1. To and Van: we changed the estimate here from $|P_n(x)|  \le n^{-B} \wedge |P_n'(x)|   \le n^{-B}$ to  $P_n(x)=0 \wedge |P_n'(x)|   \le n^{-B} $ (as was in the previous version). The reason is that we have not touched the interval $(0,1/N+1)$ at %all (see our argument at the beginning of the the next section). We can also keep this estimate, but subjecting to $1/(N+1)\le |x|\le 1$ and so on.}

%{\color{blue}2. To anh Van:  To keep the argument symmetric (between $P_n(x)$ and $P_n(-x)$, so that we could pass from $(-1,1)$ to $[0,1)$ in the proof of Theorem~\ref{theorem:expectedvalue:2}), we changed the statement of Theorem~%%\ref{theorem:repulsion:uniform} so that for the type II setting we don't need the coefficients to be identically distributed. For type I by symmetry of the distribution clearly $P_n(x)$ and $P_n(-x)$ are identically distributed. Also,  the $p_{\pm 1}$ in Theorem~%\ref{theorem:repulsion:uniform} is not the same as the $p_{\pm 1}$  previously defined for Bernoulli, so we added the last phrase for clarity. }

It is clear that in the type II setting we have $p_{\pm 1}=0$. Also, it is not hard to show that $p_{\pm 1}=O(1/ n^2)$ for the type $I$ setting (see Claim~\ref{claim:anticoncentration}.)  

Our proof  will provide  more  detailed  information about the location of double and near double roots.  For instance,  in the discrete case, we  
will show that with overwhelming probability, any  double  (or near double) root, has to be very close to either 1 or $-1$.  
We present the precise statements in  Section \ref{section:prelim} and 
with proofs  in Sections  \ref{section:repulsion:general}-\ref{section:repulsion:edge}.  As an application, we verify Theorem \ref{theorem:expectedvalue:2} in Section \ref{section:expectedvalue:2}. 
A few technical statements will be proved in the Appendices in order to maintain the flow of the presentation. 

\begin{remark} Parallel to this paper,  Peled et. al. \cite{PSZ} proved for random polynomials with the atom variable $\xi$ having support on $\{0, \pm 1\}$ that the probability of having double roots (both real and complex) is dominated by the probability of having double roots at $0, \pm 1$. The method used in \cite{PSZ} is specialized for the support consisting of $\{0, \pm 1\}$ and is totally different from the method we used in this paper. 
\end{remark}

%{\color{blue} To Hoi: May I suggest separating Theorem~\ref{theorem:repulsion:uniform} into two theorems according to type I and type II, since it seems that we have a stronger result for type I (with $|P_n(x)| \le n^{-B}$ inside the probability). Or may be you could also comment here that the estimate for type I is stronger in the above sense.}

\section{More precise statements }\label{section:prelim}

The harder case in our study is the discrete one ($\xi$ is of Type I). In  what follows we first discuss our approach for this case.

\subsection{$\xi$ of type I}  Note that the real roots of $P_n(x)$ have absolute value bounded from above by $N+1$ and below by $1/(N+1)$. It follows that we only need to consider $1/(N+1)<|x|<N+1$. Since $P_n(-x)$ and  $x^n P_n(1/x)$ have the same distribution as $P_n(x)$, it suffices to consider $1/(N+1)<x\le 1$.

 Let $0<\ep<2$ be a constant and set
\begin{equation}\label{eqn:I0I1}
I_0 := (1/(N+1), 1-n^{-2+\ep}],  \mbox{ and } I_1:= (1-n^{-2+\ep}, 1].
\end{equation}

%{\color{blue}To Hoi: We'd like to have $\epsilon<2$ so that $I_1 \subset [0,1]$ and $I_0$ makes sense.}

%{\bf I include the end point $1$; hope that is not a problem.} 

Theorem~\ref{theorem:repulsion:uniform} is a consequence of the following two results, where $\xi$ is assumed to have type $I$. 

\begin{theorem}[No double root in the bulk]\label{theorem:prelim:1}
Let   $C>0$ be any constant. Then there exists $B>0$ depending on $C$ and $\ep$ such that the following holds  
$$\P\Big( \exists x \in I_0: |P_n(x)|\le n^{-B}, |P_n'(x)| \le n^{-B} \Big) =O(n^{-C}).$$
\end{theorem}

\begin{theorem}[No double root at the edge]\label{theorem:close:Ber}  For sufficiently small $\ep$, there exists $B>0$ such that 
$$\P\Big( \exists x \in I_1: |P_n(x)|\le n^{-B}, |P_n'(x)| \le n^{-B} \Big)  =  \P\Big( P_n(1)=P_n'(1)=0 \Big).$$
\end{theorem}

In fact, it follows from the proof  in Section \ref{section:repulsion:edge} that one can take $\ep\le 1/8$ and $B=16$.
 
%{\color{blue} To anh Van: we added an explicit value for $B$ and $\ep$ here.}

%{\bf: To Yen and Hoi: It seems to me in this range of $x$ $B$ can be anything ?. In fact one can replace $n^{-B}$ by $1/100$ ? The argument goes like this: $P(1) =small thing - small thing$. The first small thing goes to 
%zero by large deviation.  The second one is $n^{-B}$. But all we need is $|P(1)| < 1$ to make it zero, so basically can we replace $n^{-B}$ by $.99$ ? } 

%{\color{blue}To Hoi: I rearrange the order of the theorems for more clarity.}

To prove Theorem \ref{theorem:expectedvalue:2}, we will need the following stronger version of Theorem~\ref{theorem:prelim:1}.

\begin{theorem}\label{theorem:repulsion:general} Assume that $\xi$ has uniform discrete distribution with parameter $N$, where $N$ is fixed. Let  $C>0$ be any constant. Then there exists $B>0$ depending on $C,N$ and $\ep$ such that the following holds with probability at least $1-O(n^{-C})$. 
\begin{enumerate}
\item[(i)] (Near double roots) There does not exist  $x \in I_0 $ such that  
$$|P_n(x)|\le n^{-B} \wedge |P_n'(x)| \le n^{-B}.$$

\item[(ii)] (Repulsion) There do not exist  $x,x'\in I_0$ with $|x-x'|\le n^{-B}$ such that  
$$P_n(x)=P_n(x')=0.$$

\item[(iii)] (Delocalization)  For any given $a \in I_0 $, there is no $x$ with $|P_n(x)|\le n^{-B}$ and $|x-a| \le n^{-B}$. 
\end{enumerate}
\end{theorem}

We remark that Theorem \ref{theorem:repulsion:general} might continue to hold for an interval larger than $I_0$, but we do not try to pursue this matter here. We next turn to the continuous case.

\subsection{$\xi$ of type II} 

For any interval $I\subset \R$, using H\"older's inequality we have
$$P(\xi\in I) = O(|I|^{1-1/p}) \ \ .$$

%{\color{blue} To Hoi: I change a bit here to reflect that we now only have $p$-integrability where $p$ could be close to $1$.}

Since $p>1$, it follows that $P(|\xi| < n^{-C}) = O(n^{-C(p-1)/p})$. Additionally, as $\xi$ has bounded $(2+\ep_0)$-moment, we have $\P(|\xi|> n^C) = \P(|\xi|^{2+\ep_0}> n^{(2+\ep_0)C}) = O(n^{-(2+\ep_0)C})$. Therefore with a loss of at most $O(n^{-C}) (C>1)$ in probability one can assume that 
\begin{equation}\label{eqn:cont:bound}
n^{-C_1}\le |\xi_i| \le n^{C_1}, \forall{1\le i\le n} \ \ ,
\end{equation}
where $C_1$ is a finite constant depending on $p$ and $C$.

%{\bf I do not see this, does $C_1$ depend on $C$ as well? Yes, $C$ is included.} 

Conditioning on this, it can be shown easily that if $|x| \le \frac 1 4 n^{-2C_1}$ then $|P_n(x)| \ge n^{-C_1}/2$ and if $|x| \ge 4 n^{2C_1}$ then $|P_n(x)| \ge n^{-C_1}x^n/2 \gg 1$. It follows that Theorem~\ref{theorem:repulsion:uniform} follows from 
the following analogue of Theorem \ref{theorem:repulsion:general}. We remark that in this theorem, we allow a more general setting where the coefficients $\xi_i$ are not necessarily iid, which is convenient in the proof. 

\begin{theorem}\label{theorem:repulsion:general'} Assume that $\xi_0, \dots, \xi_n$ have type II distributions with uniform implicit constants. 
Consider $P_n (x) = \xi_n x^n +\dots  +\xi_0$. 
Let  $C>1$ be any constant. Then there exists $B>0$ depending on $C,\ep_0$ and $p$ such that (i),(ii),(iii) of Theorem \ref{theorem:repulsion:general} hold with probability at least $1-O(n^{-C})$ with $I_0$ replaced by $I_0':=[-4n^{2C_1}, 4n^{2C_1}]$.  In other words,

 \begin{enumerate}
\item[(i)]  there does not exist  $x \in I_0'$ such that  
$$|P_n(x)|\le n^{-B} \wedge |P_n'(x)| \le n^{-B};$$

\item[(ii)] there do not exist  $x,x'\in I_0'$ with $|x-x'|\le n^{-B}$ such that  
$$P_n(x)=P_n(x')=0;$$

\item[(iii)]  for any given $a \in I_0' $, there is no $x$ with $|P_n(x)|\le n^{-B}$ and $|x-a| \le n^{-B}$. 
\end{enumerate}
\end{theorem}

%{\color{blue} To Yen: I added another subsection here.}

In the next section, we discuss the strategy to  prove Theorems ~\ref{theorem:repulsion:general} and \ref{theorem:repulsion:general'}.

\section{The general strategy and the proof of Theorem   \ref{theorem:repulsion:general} }\label{section:repulsion:general}

%{\color{blue}To Hoi: I moved the stategy section down here since it is part of the proof of Theorem~\ref{theorem:repulsion:general}, and some minor changes are made as well.}

%\subsection{Proof strategy}

In this section, we   first explain our strategy to prove (the harder) Theorem \ref{theorem:repulsion:general} and then deduce Theorem \ref{theorem:repulsion:general'} from this approach.
The rest of the proof of Theorem  \ref{theorem:repulsion:general} follows in the next section. 

Our general strategy is to reduce  the  event of having double (or near double) roots   to the event that a certain random variable takes value in a small interval. 
The key  step is to bound the probability of the latter, and here our main tool will be  a recently developed machinery, the so-called Inverse Littlewood-Offord theory
 (see \cite {NgV-advances}  for an introduction).

 Divide the interval $I_0$ into subintervals of length $\delta$ each (except for possibly the right-most interval which has the same right endpoint as $I_0$), where $\delta$ to be chosen sufficiently small (polynomially in $n$), and $B$ is then chosen large enough so that   $\delta^2 \gg n^{-B}$.

{\bf Near double roots.} Assume that there exists a subinterval $I$ and an element $x\in I$ such that $|P_n(x)|\le n^{-B}$ and $|P_n\rq{}(x)|\le n^{-B}$, then for $x_{I}$, the center of $I$, we have  
$$|P_n(x_{I})|\le \delta |P_n\rq{}(y)| + n^{-B}$$ for some $y\in I$. 

%{\bf The big O here depend on $N$, correct ? Yes} 
In the following, the implicit constants in $O(.)$ may depend on $N$ unless otherwise specified.

On the other hand, as $|P_n\rq{}(y)|\le \delta |P_n\rq{}\rq{}(z)|+ n^{-B}$ for some $z\in I$. From here, by the trivial upper bound $O(n^3)$ for the second derivative, we have
\begin{equation}\label{eqn:bound}
|P_n(x_{I})|=O (\delta^2 n^3 +n^{-B}).
\end{equation}

{\bf Repulsion.}  Assume that  $P_n (x)= P_n (x')=0$ for some $x,x'\in I_0$ with $|x-x'|\le n^{-B}$. Then there is a point $y$ between $x$ and $x'$ such that $P_n'(y)=0$. Thus, for any $z$ with$|z-y|\le 2\delta$,  
$$|P'(z)| \le 2 \delta n^{3}.$$ 

There is a point $x_I$ of some subinterval $I$ such that $|x_I-x| \le \delta$. For this $x_I$,  $|P_n(x_I)| = |x_I-x| | P'_n (z)| $ for some $z$ between $x$ and $x_I$. Because $x$ has distance at most $n^{-B} \ll \delta$ from $x'$, $x$ also has distance at most  $\delta$ from $y$, and so $z$ has distance at most $2\delta$ from $y$. It follows that
\begin{equation}\label{eqn:bound'}
 |P_n (x_I) |   \le 2 \delta^2 n^{3}.
\end{equation}

\begin{remark}\label{remark:continuous} One can also show that the {\bf repulsion} property is a direct consequence of  the {\bf near double roots} property  by choosing $B$ slightly larger if needed. Indeed, suppose that $P_n(x)=P_n(x')=0$ for some $x,x'\in I_0$ with $|x-x'|\le n^{-B}$, then consider the $y$ obtained as above. Thus $P'(y)=0$, and by using the trivial bound $O(n^2)$ for the derivative, 
$$|P_n(y)|= |P_n(y) - P_n(x)| =O(|x-y|n^2) =O(n^{-B+2}).$$
\end{remark}

%{\color{blue} To Yen, I added the remark here. }

{\bf Delocalization.} Assume that $P_n (x) =0$ and $|a-x| \le  n^{-B} \le \delta^2$, then $|P_n(a) | = |a-x| |P_n' (y)| $ for some $y$ between $a$ and $x$. On the other hand, $|P_n'(y) | \le n^{2}$ for any $y \in [0,1]$, it follows that 
\begin{equation}\label{eqn:bound''}
|P_n (a)| \le   n^{3} \delta^2.
\end{equation}

To prove  Theorem \ref{theorem:repulsion:general}, we will  show that the probability that \eqref{eqn:bound}, \eqref{eqn:bound'},\eqref{eqn:bound''} hold for any fixed point $x$ of $I_0$ is $O(\delta n^{-C})$. This definitely takes care of \eqref{eqn:bound''} and hence (iii) of Theorem~\ref{theorem:repulsion:general}. Since there are $O(\delta^{-1})$ subintervals $I$, by the  union bound we also obtain (i) and (ii) of Theorem~\ref{theorem:repulsion:general}.

%{\bf The indexing are somewhat confusing, maybe we want to change (1)(2) of Theorem A to (i)(ii) of Theorem  A. (Changed)} 

%{\color{blue} {\bf To Hoi:}  In \eqref{eqn:bound''} $a$ is not necessary a center of a subinterval. Also the choice of $\delta$ should also ensure that $\delta'=\delta^2$ dominates $n^{-B}$ (which was used for delocalization as above). I  summarized the argument below in Lemma~\ref{l.almost-zero}, so that we could simply refer to it later in the proof.}

In fact we will show the following stronger estimate

\begin{lemma} \label{l.almost-zero} Assume that $\xi$ has type I. Then there is a constant $c>0$ which depends only on $N$ and $\ep$ such that for every $A>0$ sufficiently large the following holds for  every $0<C_1\le C_2$ and $C_1n^{-A} \le \delta \le C_2 n^{-A}$ 
$$\sup_{x \in I_0} \P\Big(|P_n(x)| \le \delta^2\Big)  = O(\delta^{1+c}),$$
here the implicit constant may depend on $N$ and $c$ and $C_1$ and $C_2$.
\end{lemma}

%{\color{blue}To Hoi: I think that the implicit constant does not depend on $A$ since we have already require $A$ to be large, but please double check.}

%{\bf What does ~ mean ?} 

(The fact that we have an extra factor of $n^3$ or $n^2$ in \eqref{eqn:bound}, \eqref{eqn:bound'},\eqref{eqn:bound''} is not an issue here, since these powers could be included  as part of the $\delta$ of Lemma~\ref{l.almost-zero}.)

Note that by making $c$ slightly smaller it suffices to prove the Lemma for $C_1=C_2=1$, i.e. $\delta=n^{-A}$, which we will assume in the following. We will justify this key lemma in the next section. In the rest of this section, we apply our argument to handle distributions of Type II. 

\subsection{Proof of  Theorem \ref{theorem:repulsion:general'}} By following the same argument, and by \eqref{eqn:cont:bound}, it is enough to show the following analog of Lemma \ref{l.almost-zero} for 
Type II variables. Recall that we are working under the assumption that $\xi_0,\dots, \xi_n$ are uniformly Type II and independent, but they are not required to be identically distributed.

\begin{fact} \label{f.almost-zero} There is a constant $c>0$ which depends only on $p$  such that for every $A>0$ sufficiently large the following holds for  $\delta = n^{-A}$ 
$$\sup_{x \in I_0'} \P\Big(|P_n(x)| \le \delta^2\Big)  = O(\delta^{1+c}).$$
\end{fact}

Thanks to the analytic properties of Type II variables, this statement is much easier to prove than Lemma \ref{l.almost-zero}; the details follow. 

%{\color{blue}To anh Van: we modified the proof here a bit to allow for the more general setting when $\xi_j$'s are not identically distributed.}

\begin{proof}[Proof of Fact \ref{f.almost-zero}]For any $I\subset \R$, by H\"older's inequality 
$$P(\xi_0\in I) = O(|I|^{1-1/p}) \ \ .$$
Thus, by conditioning on $\xi_1,\dots,\xi_n$, for any $x$ we have
$$\P(|\sum_{i=0}^n \xi_i x^i|\le \delta^2)=\P(-\delta^2 -  \sum_{i=1}^n \xi_i x^i\le \xi_0 \le \delta^2-  \sum_{i=1}^n \xi_i x^i) = O(\delta^{2(1-1/p)}).$$
The desired conclusion follows immediately if $p>2$. To handle the general case, let $\rho_j$ denote the density of the distribution of $\xi_j$, which is $p$-integrable for $p>1$ by the given assumption. Since $\int \rho_j(x)dx=1$, it follows immediately via convexity that  $\rho_j$ is also $q$-integrable for every $q \in [1,p]$ and furthermore 
$$\sup_j \|\rho_j\|_q = O_q(1)$$
thanks to the fact that $\xi_j$'s are uniformly Type II. For convenience, let $C_q$ denote the right hand side in the estimate above.

Now, let $k$ be a large integer that depends only on $p$ such that $k/(k-1) < p$. By Young's convolution inequality and an induction over $k$, it is clear that for any family of functions $g_0, \dots, g_{k-1}$ we have
$$\|g_0 \ast \dots \ast g_{k-1} \|_{\infty} \le \prod_{j=0}^{k-1} \|g_j\|_{k/(k-1)}.$$
Consider the random variable $R_k = R_k(x)= \xi_0 + x\xi_1+\dots + x^{k-1}\xi_{k-1}$.  Since $\xi_j$'s are independent,  the density of $R_k$ (which we will denote by $r_k$) equals to the convolution of the density of $\xi_0$, $x\xi_1$, \dots, $x^{k-1}\xi_{k-1}$. Let $g)j$ denote the density of $x^j \xi_j$, clearly $g_{j}(t) = x^{-j} \rho_j(t/x^j)$, and
\begin{align*}
\|g_j\|_{q} &= (\int x^{-jq} |\rho_j(t/x^j)|^qdt)^{1/q}\\
&=x^{-j(q-1)/q} \|\rho_j\|_q.
\end{align*}

Consequently,
\begin{align*}
\|r_k\|_\infty &= \|g_0 \ast \dots \ast g_{k-1}\|_\infty \le \prod_{j=0}^{k-1} x^{-j/k} \|\rho_j\|_{k/(k-1)}\\
&\le (C_{k/(k-1)})^k x^{-(k-1)/2} = O_k(x^{-(k-1)/2}).
\end{align*}

Therefore for every $x\in I_0'$ we have
$$\|r_k\|_\infty=O(n^{C_2}),$$
where $C_2$ is a finite constant depending only on $p$ and $C_1$.

Now, for every $n\ge k$ ( recall that  $k$ is a constant) and $x\in I_0'$ we have
\begin{align*}
\P(|P_n(x)| \le \delta^2) &= \P(-\delta^2 -  \sum_{i=k}^n \xi_i x^i\le R_k(x) \le \delta^2-  \sum_{i=k}^n \xi_i x^i)\\ 
&= O(\delta^2 \|r_k\|_\infty) = O(n^{C_2}\delta^2).
\end{align*}

Thus by choosing   $A$ sufficiently large we obtain the desired estimate (with $\delta=n^{-A}$ and any $c<1$).
\end{proof}

%{\color{blue} To Hoi: I added the proof to allow for every $p>1$ here.}
%{\color{blue} To Yen: I added a short proof here.}

%{\color{blue} To Yen: I added a remark here. As the proof of Lemma~\ref{l.almost-zero} is one of our key results, perhaps it is worth a separate section?}

%{\color{blue} To Hoi: I generalize Remark~\ref{remark:gaussian} a little bit. 

%A question for you: can you prove Theorem~\ref{theorem:close:Ber} and hence Theorem~\ref{theorem:repulsion:uniform}  and Theorem~\ref{theorem:expectedvalue:2} for the $p$-integrable setting mentioned in the remark? If so we should also state a theorem about that (unless it is a known result).

%If so, it looks interesting here since distributions with a nice (i.e. $L^{2+\epsilon}$ integrable) density and discrete distributions are two extremes of the general setting, so  may be we interpolate the argument and get a generic theorem that would imply both? Something reminiscence of the Radon--Nykodym theorem where a (reasonable) measure could be decomposed into an absolutely continuous part and a counting measure part. On the other hand, so far we could only handle near uniform discrete distributions. So may be this would be a subject of future study.}

\section{Proof of Lemma~\ref{l.almost-zero}:  bounds on small value probability for $P_n$}

By making $c$ smaller if necessary, it suffices to prove the Lemma for $\delta=n^{-A}$ where $A$ is sufficiently large. Also, as indicated before, $B$ will be chosen such that $\delta^2 \gg n^{-B}$ (for instance $B=2A+10$). Fix $x\in I_0$, all the implicit constants below are independent of $x$. 

%{\bf Is it better if somewhere at the beginning of all proofs, we fix, once and for all, $\delta = n^{-A}$ and $B= 2A +10$ or something like that; then we do not need to repeat the above sentence, and the dubious $\delta \sim n^{-A}$ ? } 

%{\color{blue}{\bf To Hoi:} I updated the proof below a little bit  to connect with the last section, and also because it was previously written for the centers of the dividing subintervals (which does not include the delocalization case).}

We divide $I_0$ into $(1/(N+1),1-\log^2 n/n] \cup (1-\log^2 n/n,1-n^{-2+\ep}]$ and prove the lemma for $x$ inside each interval separately.  For the first interval $(1/(N+1),1-\log^2 n/n]$, we will present a proof for the Bernoulli case (i.e. $N=1$) first to demonstrate the main ideas, and then modify the method for uniform distributions later. Our treatment for $ (1-\log^2 n/n,1-n^{-2+\ep}]$ works for both settings.

\subsection{Proof for $1/(N+1)<x\le 1-n^{-1}\log^2 n$, the Bernoulli case}\label{subsection:smallx:Ber}

Roughly speaking, the proof exploits the lacunary property of the sequence $\{1,x,\dots,x^n\}$ in this case.   

Let $\ell \in \mathbb Z$ be  such that 
$$x^{\ell}< 1/2 \le x^{\ell-1}.$$ 

As $x\le 1-n^{-1}\log^2 n$, we must have 
$$\ell =O(n/\log^2 n)=o(n).$$  

Note that if $x < 1/\sqrt{2}$ then $\ell=2$. As the treatment for this case is a bit more complicated, we postpone it for the moment. In the sequel we assume that $x >1/\sqrt{2}$, and thus $\ell \ge 3$. 

{\bf Treatment for $1/\sqrt{2}\le x \le 1-n^{-1}\log^2 n$.} Let $k$ be the largest integer such that 
$$x^{\ell k} \ge \delta^2\equiv n^{-2A}$$ 

In other words,
$$k = \lfloor\frac{(2A)\log n}{\ell\log (1/x)}\rfloor.$$
where $\lfloor x \rfloor$ denote the largest integer that does not exceed $x$.

As $x \le 1-n^{-1}\log^2 n$ and $n$ is sufficiently large, it follows that  $k \ell$ is strictly less than $n$ and $k$ is at least $\Omega(\log n)$, one has the following trivial bound 
$$k\ge 10.$$

We say that a finite set $X$ of real numbers is $\gamma$-separated if the distance between any two elements of $X$ is at least $\gamma$.

\begin{claim}\label{claim:separation:1}
The set of all possible values of $\sum_{1\le j\le k} \ep_j x^{j\ell}, \ep_j\in \{-1,1\}$ is $2x^{k\ell}$-separated. 
\end{claim}

\begin{proof} (of Claim \ref{claim:separation:1})
Take any two elements of the set. Their distance  has the form  $2|\ep_{m_1} x^{m_1\ell} + \dots+ \ep_{m_j} x^{m_j\ell}|$ for some $1\le m_1<\dots<m_j \le k$. As $x^{\ell}<1/2$, this distance  is more than $2x^{k\ell}$.
\end{proof}

Using Claim \ref{claim:separation:1}, we have
\begin{equation}\label{eqn:bound:condition}
\sup_R\P_{\xi_{j\ell},1\le j\le k}(|\sum_{j=1}^k \xi_{j \ell} x^{j \ell}+R| \le x^{k \ell})\le 2^{-k}.
\end{equation}

By conditioning on other coefficients $\xi_m$'s i.e. $m\not\in \{\ell,2\ell,\dots,k\ell\}$, it follows that
\begin{align}\label{eqn:bound:gain}
%\P(|P_n(x)|&\le n^{-2A-3} ) \le 2^{-k} \nonumber \ \ , \\
\P(|P_n(x)|&\le \delta^2 ) \le 2^{-k}.
\end{align}

Recall that $x^{\ell}< 1/2 \le x^{\ell-1}$. Using the fact that $\ell\ge 3$ and $k\ge 10$, we obtain $(\ell-1)k \ge \frac{3}{5}\ell (k+1)$. It follows that
$$2^{-k}\le x^{(\ell-1)k} \le   x^{3(k+1)\ell/5} \le (\delta^2)^{3/5} = \delta^{6/5}.$$
Therefore
$$\P(|P_n(x)|\le \delta^2) = O(\delta^{6/5})$$
as desired.  This completes the treatement of the case $1/\sqrt{2}\le x \le 1-n^{-1}\log^2 n$.

{\bf Treatment for $1/2+c_0 <x <\sqrt{1/2}$. } Let $c_0$ be a small positive constant. We  show that the treatment above also carries over for this range of $x$ with a minor modification. 

%{\color{blue}To Hoi: I changed $\ep$ to $c_0$ since we already used $\ep$ in the definition of the interval $I_0$.}

As $x< \sqrt{1/2}$ and $x^{\ell}< 1/2 \le x^{\ell-1}$,  we must have $\ell=2$ for all $x$ in this range. Recall that the integer $k$ was chosen so that
$$x^{2(k+1)} < \delta^2 \le x^{2k}  \ \ .$$
By following Claim \ref{claim:separation:1}, we again arrive at \eqref{eqn:bound:condition} and \eqref{eqn:bound:gain}.

Now, as $x \ge 1/2+c_0$, we have $1/2 \le x^{1+c_1}$ for some small positive constant $c_1$ depending on $c_0$. As such, using the fact that $k$ has order $\log n$, we have $\frac k{k+1} \ge \frac{1/2+c_1/4}{1/2+c_1/2}$ for $n$ sufficiently large. It follows that
\begin{align*}
2^{-k}\le x^{(1+c_1)k}&=(x^{2k})^{1/2+c_1/2} \\
&\le  (x^{2(k+1)})^{1/2+c_1/4}    <  (\delta^2)^{1/2 +c_1/4}  \ \ .
\end{align*}

We obtain
$$\P(|P_n(x)| \le \delta^2) = O( \delta^{1+c_1/2} )$$
as desired.

{\bf Treatment for $1/2<x < 1/2+c_0$.} Recall that in this case $\ell=2$. For this range of $x$ we introduce the following improvement of Claim \ref{claim:separation:1}.

\begin{claim}\label{claim:separation:2} Assume that $1/2<x<1/2+c_0$ and $c_0$ is sufficiently small. Then the set of all possible values of $\sum_{j=0}^{k}\ep_{2j} x^{2j} + \sum_{j=0}^{\lfloor  k/8 \rfloor} \ep_{8j+1}x^{8j+1}, \ep_i\in \{-1,1\}$, is $x^{2k}/8$-separated.
\end{claim}

\begin{proof} (of Claim \ref{claim:separation:2})

The distance between   any two terms is at least
$$2x^{2k}[1-x-\sum_{j=1}^{\infty} x^{2j} - \sum_{j=0}^{\infty} x^{8j+1}] > x^{2k}/8,$$
where we used the fact that the factor within the bracket is at least $1/16$, provided that  $c_0$ is chosen  sufficiently small.
 \end{proof}

Using Claim \ref{claim:separation:2}, we obtain the following slight improvement of \eqref{eqn:bound:condition}
\begin{equation}\label{eqn:bound:condition:improvement}
\sup_R\P_{\xi_{2j},0\le j\le k}(|\sum_{j=0}^{k} \xi_{2j} x^{2j}+\sum_{j=0}^{\lfloor k/8 \rfloor} \xi_{8j+1}x^{8j+1} +R| \le x^{2k}/8)\le 22^{-k-\lfloor k/8 \rfloor} \le 2^{-9k/8+2}.
\end{equation}

Now, as $k$ is order $\log n$, by taking $n$ large we have
$$2^{-9k/8}\le x^{9k/8} \le x^{10(k+1)/9} \le (\delta^2)^{5/9} = \delta^{10/9} \ \ ,$$
which implies the desired conclusion.
 
%{\bf In the definition of "separation", we use distance between elements of a set, so we should  use this phrase in the proofs, rather than difference between terms in a sequence. I tried to change all such phrases, but there may be left-overs, please check !} 

\subsection{Proof for $1/(N+1)<x\le 1-n^{-1}\log^2 n$, the uniform case}\label{subsection:smallx:uniform}

The $N=1$ case was treated before, so we only consider $N>1$. 
 
As before, let $\ell$ be integer such that
\begin{equation}\label{e.ell-choice}
x^{\ell} < \frac 1{2N+1} \le x^{\ell-1} \ \ .
\end{equation}
Since $x>1/(N+1)$, it follows that $\ell \ge 2$, and 
$$\ell = O_N(\frac n{(\log n)^2}).$$

Let $k$ be the largest integer such that
$$x^{\ell k}  \ge \delta^2 =n^{-2A}$$

We first show that $k \ge \Omega_N(\log n)$ while $\ell k = o_N(n)$. In deed,  by definition we have
$$k  = \lfloor  \frac{2A\log n}{\ell  \log (1/x)} \rfloor.$$
 Since  $\log (1-a) < -a$ for every $a\in (0,1)$, it follows that
$$\log (1/x) \ge - \log(1-\frac{(\log n)^2}n) \ge \frac{(\log n)^2}n$$
therefore
$$k \le 1+ O(\frac n{\ell  \log n})$$
Since $\ell  = o(n)$, it follows that $k  \ell  <n$ for $n$ sufficiently large. Furthermore,   it follows from \eqref{e.ell-choice} that
$$\frac 1{2N+1}>x^{\ell}  \ge  \frac {x}{2N+1} \ge \frac 1{(N+1)(2N+1)}.$$
Therefore
$$0<\ell  \log (1/x)  \le O_N(1).$$
Hence $k \ge c_N \log n$ for some $c_N$ depending only on $N$. 

Consider the sequence $\sum_{1\le j\le k}\ep_j x^{j\ell}$ where $\ep_j\in \{\pm 1,\pm 2,\dots, \pm N\}$.  We'll show the following separation property:

\begin{claim}\label{cl.separation}
The set of all possible values of  $\sum_{1\le j\le k}\ep_j x^{j\ell}$, where $\ep_j\in \{\pm 1,\pm 2,\dots, \pm N\}$, is $x^{k\ell}$-separated. 
%For any fixed $R$ it holds that
%$$P(|R+Q(\alpha)|\le \alpha^{k_\alpha \ell_\alpha}) \le 3(\frac 1{2N})^{k_\alpha}$$
\end{claim}
\proof[Proof of Claim~\ref{cl.separation}] Take any two terms of the sequence. Consider their difference, which has the form $b_{m_1}x^{m_1\ell}+\dots + b_{m_j} x^{m_j\ell}$
for some $1\le m_1<\dots<m_j \le k$, and $|b_{m_1}|\ge 1$ and $|b_{m_2}|,\dots, |b_{m_j}| \le 2N$. As $x^{\ell} <1/(2N+1)$, this difference is more than $x^{k\ell}$. \endproof

It follows that for every $R$
$$\P(|\sum_{j=1}^k \xi_{j\ell} x^{j\ell} + R| \le x^{kl}) \le 3 (\frac 1{2N})^k.$$

Now, using the independence of $\xi_0,\dots, \xi_n$ and by conditioning on $\xi_j$'s with $j\not\in \{\ell,2\ell,\dots,k\ell\}$,
we obtain
$$\P(|P_n(x)| \le x^{k \ell}) \le 3 (\frac {1}{2N})^{k}.$$

Thus, by the choice of $k$, we obtain the key bound
\begin{equation}\label{eqn:uniform:key}
\P(|P_n(x)| \le \delta^2) \le 3 (2N)^{-k}.
\end{equation}

Next, consider two cases:

{\bf Case 1: $1/\sqrt{2N+1} \le x \le 1-( \log n)^2/n$.} Since $x^{\ell}<\frac 1{2N+1}\le x^{\ell-1}$, it follows that $\ell\ge 3$. Thus,
\begin{align*} 
(2N)^{-k} &= (1/(2N+1))^{k  \log(2N)/\log(2N+1)} \\
&\le x^{(\ell -1)k  \log(2N)/\log(2N+1)}.
\end{align*}

Let $\gamma_N:=\frac{2\log(2N)}{3\log(2N+1)}$. As $\ell-1 \ge 2\ell/3$, we obtain
\begin{align*}
\P(|P_n(x)| \le n^{-2A})   &\le 3(x^{\ell(k+1)})^{\gamma_N k/(k+1)}\\
&\le 3(\delta^2)^{\gamma_N k /(k +1)}.
\end{align*}

Since $k$ is controlled below by some $c_N\log n$, we could make $k /(k +1)$ arbitrarily close to $1$  by taking $n$ large (independent of $\alpha$). Thus it suffices to show that
$$\gamma_N>\frac 12.$$
%(which then allows for choosing $A$ very large compared to $C>0$.)
But it is clear that this holds for every $N\ge 2$. Indeed, consider the function defined on $(0,\infty)$
\begin{align*}
f(x) &=4\log(2x)-3 \log(2x+1)\\
f'(x)&=\frac{4}{x} - \frac{6}{2x+1} = \frac{2x+4}{x(2x+1)}>0,
\end{align*}
so for $x\ge 2$ we have $f(x)\ge f(2)>0$.

{\bf Case 2: $1/(N+1)< x < 1/\sqrt{2N+1}$.} It follows that $\ell=2$. Also,
\begin{align*}
(2N)^{-k} &= (1/(N+1))^{k \log(2N)/\log(N+1)}\\
&\le x^{k \log(2N)/\log(N+1)}.
\end{align*}

Let $\beta_N = \frac{\log(2N)}{2\log (N+1)}$, it follows from \eqref{eqn:uniform:key} that
$$\P(|P_n(x)| \le \delta^2) \le 3(x^{2(k +1)})^{\beta_N k /(k +1)}.$$

By choice of $k$, $x^{2(k+1)} < \delta^2 \le x^{2k}$, therefore
$$\P(|P_n(x)| \le \delta^2) \le  3(\delta^2)^{\beta_N k/(k+1)}.$$
As before, by choosing $n$ large (independent of $x$) we could ensure that $k /(k +1)$ is arbitrarily close to $1$. Therefore it suffices to show that
$$\beta_N > \frac  12,$$
which is clear for $N>1$.

\subsection{Roots behaviour in $1-\log^2n/n \le x \le 1-n^{-2+\ep}$ for the uniform case}\label{subsection:closex} Our treatment 
of this interval is more difficult as here the terms  $x^i$ are comparable.   Our main tool is a following theorem, which is an example 
of the recently developed machinery of Inverse Littlewood-Offord theorems (see  \cite{NgV-advances} for an introduction).

\begin{theorem}\label{theorem:inverse} Fix positive integers $C,C'$ and $0<\ep_0<1$, and assume that

$$\rho=\sup_{v\in \R}\P\left(|\sum_i \xi_i v_i-a| \le \beta\right) \ge n^{-C'},$$ 

for some real numbers $v_1,\dots,v_n$, where $\xi_i$ are iid random variables of uniform distribution with fixed parameter $N$. Then for  any number $n'$ between $n^{\ep_0}$ and $n$, with $n$ sufficiently large, there exists a proper symmetric generalized arithmetic progression $Q$, that is $Q=\{\sum_{i=1}^r x_ig_i : x_i\in \Z, |x_i|\le L_i \}\subset \R$, such that

\begin{enumerate}[(i)]

\item $Q$ is $C$-proper, i.e. the elements of the set $CQ=\{\sum_{i=1}^r k_ig_i : k_i\in \Z, |k_i|\le CL_i \}$ are all distinct.
\vskip .1in

\item $Q$ has small rank, $1\le r=O (1)$, and small cardinality
$$|Q| =O(\rho^{-1} \ell_0^{1-r}) =O(\rho^{-1}),$$

where the implied constants here depend on $C,C',N$ and $\ep_0$, and  $\ell_0=\sqrt{n'/\log^2 n}$.

\vskip .1in

\item For all but at most $n'$ elements $v$ of $\{v_1,\dots,v_n\}$, there exists $q\in Q$ such that 

$$|q-v|\le T_0\beta/\ell_0,$$ 

where $T_0=\Theta(1)$ independent of $C$.

\vskip .1in

\item The number $T_0 \beta/\ell_0 \in Q$, i.e., there exist $|k_1|\le L_1,\dots$, $|k_r|\le L_r$ such that 

$$T_0 \beta/\ell_0 =\sum_i k_i g_i.$$

\vskip .1in

\item All steps $g_i$ of $Q$ are {\it integral multiples} of $T_0 \beta/\ell_0$. 
\end{enumerate}
\end{theorem}

%{\bf To Hoi: In (ii) do we have $n^{-r/2} $ or $n'^{-r/2} $  or something like that in the bound  ? We may not need it, I just want to make sure.} 

%{\color{blue} To anh Van: we added $\ell_0^{1-r}$ into the bound of $|Q|$ for clarity}.

We will provide  the deduction of Theorem \ref{theorem:inverse} from \cite[Theorem 2.9]{NgV-advances} in Appendix \ref{appendix:inverse}.
 
%{\color{blue}To Hoi: in the theorem, what is $k$ in the last item? And in the second last item, the rank $d$ should be  $r$? (I also see multiple occurences of $d$ below, they should be changed to $r$.) And we probably also want $C\in \mathbb Z$ in order for $CQ$ to be   defined in the usual additive sense?}

%{\color{blue} To Hoi: I guess another way of making $\ell_0$ independent of $\beta$ is to do the case $\beta=1$ first, then rescale the general case by dividing every $v_j$ by $\beta$. But your changes are better, more quantitative.}

It follows from the $C$-properness (i) of $Q$ that for any $t\in \Z, 0<t \le C$, the equation 
\begin{equation}\label{eqn:GAP:special}
k_1(t)g_1+\dots+k_r(t)g_r = t (T_0 \beta/\ell_0), k_i(t)\in \Z, |k_i(t)|\le C L_i
\end{equation}
has a unique solution $(k_1(t),\dots,k_d(t))= t \cdot (k_1,\dots,k_d)$.

Now fix $x\in (1-\log^2n/n, 1)$. Choose the largest $n_0$ so that $x^{n_0} \ge 1/10$, thus 
$$n_0 \ge n/\log^2 n.$$ 

In the sequel, set $\ep_0:=\ep/2$ and 
$$n':=n^{\ep_0} \mbox{ and } \beta:=\alpha (1-x)^A \ell_0,$$ 
 
 where $\alpha$ to be chosen sufficiently small depending on $A$.

% Here we note that given any $0<\epsilon'_0 < \epsilon_0$, by choosing $n$ large we will have $n_0^{\epsilon'_0}<n'<n_0$, thus the theorem is applicable to the pair $(n_0,n')$.

%{\color{blue}To Hoi: It looks to me that the $\ell_0$ that comes out of Theorem~\ref{theorem:inverse} would depend on $\beta$ (it depends on $m$ which may depends on $\beta$, as far as we can say from the statement of the Theorem). So how could you set $\beta $ to be something depending on $\ell_0$ again? }

We will prove the following crucial bound.

\begin{lemma}\label{lemma:crucial:1/2}  We have 
$$\rho=\sup_{r\in \R} \P\Big(|\sum_{i=0}^{n_0} \xi_i x^i -r|\le \beta \Big) =O( (n^{1-\ep_0}/\log^2 n)^{-A})  .$$

\end{lemma}

%{\color{blue} To Hoi: What are the range of $\epsilon_0$ and $A$ so that the Lemma holds?  We should specify that in the statement of the Lemma. }

\begin{proof}[Proof of Lemma \ref{lemma:crucial:1/2}] Without loss of generality assume $A$ is an integer. Assume otherwise that 
\begin{equation}\label{eqn:rho:large}
\rho \ge C_1 (n^{1-\ep_0}/\log^2 n)^{-A}
\end{equation}

for some sufficiently large constant $C_1$ to be chosen depending on all other parameters.

Then by Theorem \ref{theorem:inverse}, all but $n^{\ep_0}$ of the elements are $T_0\beta/\ell_0$-close to a proper GAP $Q$ of rank $r=O(1)$ and size $O(\rho^{-1})$. Furthermore, as noticed, the generators of $Q$ can be chosen to be integral multiples of $T_0\beta/\ell_0$.

 Let $I$ be the collection of indices $i\le n_0$ where $x^i$ can be well-approximated by the elements of $Q$ as stated in Theorem \ref{theorem:inverse}. Then as $|I|\ge n_0-n^{\ep_0}$, $I$ contains a discrete interval of length $ \lfloor n^{1-\ep_0}/\log^2n \rfloor - 1$, which we denote by $I_0=\{i_0,\dots,i_0-   \lfloor n^{1-\ep_0}/\log^2 n \rfloor + 2\}$. (Note that the symbol $I_0$ was used for a different interval previously, which should not be confused with the current setting.)
 
 For any $i$ from $I_0$ such that $i-A\in I_0$, consider the sequence $x^i,\dots, x^{i-A}$, together with their approximations $q_i,\dots, q_{i-A}$ from $Q$.  By the choice of $\beta$,
$$\frac{\alpha^{-1} \beta}{10 \ell_0} \le  \sum_{k=0}^{A} (-1)^{ A - k} \binom{A}{k}x^{i-k} = x^{i-A} (1-x)^{A} \le  \frac{\alpha^{-1}  \beta}{\ell_0}.$$

As such, by (ii)
\begin{equation}\label{eqn:hard:1}
\frac{\alpha^{-1} \beta }{10\ell_0}- 2^A T_0\beta/\ell_0 \le \sum_{k=0}^{A} (-1)^{ A - k} \binom{A}{k}q_{i-k} \le  \frac{\alpha^{-1} \beta }{\ell_0} + 2^A T_0 \beta/\ell_0.
\end{equation}

With the choice $\alpha=2^{-A-5}T_0$, one guarantees that $\alpha^{-1}> 10 \times 2^A T_0$, and thus the LHS of \eqref{eqn:hard:1} is strictly positive. After choosing $\alpha$, we choose $C=2^{A+10}T_0$ in Theorem \ref{theorem:inverse} so that $C> (\alpha^{-1} +2^A)T_0$, the constant in the RHS of \eqref{eqn:hard:1}.

%{\color{blue}To Hoi: Here may be it is worth emphasizing that the upper bound on the constant $T_0$ from Theorem~\ref{theorem:inverse} is absolute and independent of $C$   (otherwise the above claim is not true). I saw that in the statement of the Theorem you used the notation $T_0 = \Theta(1)$, may be you meant to say that, I still think this should be stressed  since not everyone understands that notation.}

Next, assume that $q_j= c_{j1}g_1+\dots + c_{jd} g_d$ for $|c_{j1}| \le L_1,\dots, |c_{jd}|\le L_d$. Then it follows from the choice of $C$ and from \eqref{eqn:hard:1} that 
$$0< \big(\sum_{k=0}^{A} (-1)^{A-k}\binom{A}{k}c_{i-k,1}\big) g_1 + \dots + \big(\sum_{k=0}^{A} (-1)^{A- k}\binom{A+1}{k}c_{i-k,d}\big) g_d  <C \beta/\ell_0.$$

Consequently, recalling that all the generators $g_i$ are integral multiple of $T_0\beta/\ell_0$, there exists $0<t\le C, t\in \Z$ such that 
$$ \big(\sum_{k=0}^{A} (-1)^{A-k}\binom{A}{k}c_{i-k,1}\big) g_1 + \dots + \big(\sum_{k=0}^{A} (-1)^{ A-k}\binom{A+1}{k}c_{i-k,d}\big) g_d = t \beta/\ell_0.$$

It thus follows from \eqref{eqn:GAP:special} that 
$$\sum_{k=0}^{A} (-1)^{A-k}\binom{A}{k}c_{i-k,1}= tk_1 \wedge \dots \wedge \sum_{k=0}^{A} (-1)^{A-k}\binom{A}{k}c_{i-k,d} =tk_d.$$

In summary, we obtain the following key property for all $i\in I_0$ and $i\ge A$,
$$\sum_{k=0}^{A} (-1)^{A-k} \binom{A}{k}c_{i-k,1} \in \{k_1,\dots, Ck_1 \} \wedge \dots \wedge \sum_{k=0}^{A} (-1)^{A-k} \binom{A}{k}c_{i-k,d}\in \{k_d,\dots, Ck_d\}.$$

As $k_1,\dots,k_d$ cannot be all zero, without loss of generality, assume that $k_1>0$. Thus for every $i\in I_0$ such that $i-A \in I_0$ we have
\begin{equation}\label{eqn:hard:2}
1\le  k_1 \le \sum_{k=0}^{A} (-1)^{A-k}\binom{A}{k}c_{i-k,1}  \le C k_1.
\end{equation}
 
 We next require the following observation. 
 
%{\color{blue}To Hoi: We may need an updated version of the claim below, and we should explain how to apply it to our current context (i.e. specify values of $x_{i}$ in the application).}

\begin{claim}\label{claim:increasing}
Assume that  $\{x_i\}_{i=0}^m$ is a sequence of real numbers which satisfy the following inequality for all $A\le i\le m$
$$
1\le \sum_{k=0}^{A} (-1)^k \binom{A}{k}x_{i-k}.
$$

Then there exist $0\le i,j\le m$ such that 
$$|x_i-x_j|\ge C_A m^{A},$$
where $C_A>0$ depends on $A$. 
\end{claim}

\begin{proof} (Proof of Claim \ref{claim:increasing})
Define
$$\Delta^{0} (x_i) := x_i  \mbox{ and } \Delta^{k} (x_i) := \Delta^{k-1} (x_{i-1}) - \Delta^{k-1} (x_{i} ).$$ 

By the assumption and  Pascal's triangle identity
$$1\le  \sum_{j=0}^A  (-1) ^j { A \choose j} x_{i-j}=\Delta ^A (x_i), \forall {A\le i\le m}.$$

It follows  that $\Delta^{A-1} (x_{i-1}) \ge \Delta ^{A-1} (x_{i} ) + 1$. Thus, there are at least $(m-A)/4$ consecutive indices  $i\ge A$ such that the corresponding 
$\Delta ^{k-1} (x_i)$ have the same signs and absolute value at least $(m-A)/4$.  Without loss of generality, we can assume that all of them are at least $(m-A)/4$. Repeat the argument with $A-1$ and the above subsequence. After $A$ repetitions, we obtain a sub interval $I$ of length $(m-A)/4^A$ of  $[A,m]$, where all $x_i, i \in I$ have the same signs and absolute value at least $((m-A)/4)^A$.  \end{proof}

Applying Claim \ref{claim:increasing} with $m=\lfloor n^{1-\ep_0}/\log^2 n\rfloor -2$ and $\{x_i\}_{i=0}^m:= \{c_{i_0-i,1}\}_{i=0}^m$, we obtain $L_1\ge C_A m^{A}$, and so
$$|Q| \ge C_A(n^{1-\ep_0}/\log^2n)^A .$$

This contradicts with the bound $|Q| =O(\rho^{-1})$ from Theorem \ref{theorem:inverse} and with the bound $\rho\ge C_1(n^{1-\ep_0}/\log^2 n)^{-A}
$ from \eqref{eqn:rho:large} because $C_1$ is sufficiently large. This completes the proof of Lemma \ref{lemma:crucial:1/2}.
\end{proof}

Now we conclude the subsection by proving Theorem \ref{theorem:repulsion:general} for the interval $(1-\log^2n/n,n^{-2+\ep}]$. 

%{\color{blue}To Hoi: I update the proof below to be consistent with the setup at the begining of the current Section (i.e. Section~\ref{section:repulsion:general}).}

\begin{proof}[Proof of Lemma~\ref{l.almost-zero}  for $1-\log^2n/n  <x \le 1- n^{-2+\ep}$] We need to show the existence of $c>0$ (that depends only on $\ep$ and $N$) such that for every $A>0$ sufficiently large
$$P(|P_n(x)| \le \delta^2) = O(\delta^{1+c}).$$

%where we recall that $\delta  = n^{-A}$. 

We will apply Lemma~\ref{lemma:crucial:1/2} with $\ep_0= \ep/2$, and let $\alpha$ be the corresponding constant. By making $c$ smaller if necessary, it suffices to prove Lemma~\ref{l.almost-zero} for
$$\delta= \sqrt{\alpha} n^{-A(1-\ep/2) +\ep_0/4}$$
where $A$ is sufficiently large (instead of requiring $\delta=n^{-A}$ as before.) 

As $x \in (1-\log^2n/n,1- n^{-2+\ep}]$,  one can verify that 
$$\alpha (1-x)^A \ell_0  \le \alpha (1-x)^A \sqrt{n'} \le \alpha n^{-A(2-\ep)+\ep/2} = \delta^2.$$ 

Thus, by Lemma \ref{lemma:crucial:1/2},
\begin{align*}
\P\Big(|P(x)| \le \delta^2 \Big) \le \sup_{r\in \R} \P\Big(|\sum_{i=0}^{n_0} \xi_i x^i -r|\le \alpha (1-x)^A \ell_0 \Big) &=O\big( (n^{1-\ep_0}/\log^2 n)^{-A}\big)\\
&= O(n^{-A(1-3\ep/8)})\\
&=O(\delta^{1+c}) \ \ ,
\end{align*}
for sufficiently small $c$, provided that   $A$ is sufficiently large.  
\end{proof}

%{\color{blue} To Yen: I have just restated Theorem \ref{theorem:inverse} a bit, adding a small part in Appendix B explaining how to get it from the original version. Frankly, the treatment of this part took me the most time. As in this case, when the $x^i$ are comparable, an anti-concentration like $\rho$ seems useful, but perhaps there are simpler/more direct  ways to handle this part without using \cite{NgV-advances} at all.}

\section{Proof of Theorem \ref{theorem:close:Ber}: no double root at the edge} \label{section:repulsion:edge} Set $t:=1-x$. Then $0\le t \le n^{-2+\ep}$. For every $i=0,1,\dots, n$, write
$$x^i=(1-t)^i = \sum_{0\le k\le n} (-1)^k \binom{i}{k}t^k,$$
where $\binom{i}{k}=0$ if $i<k$. Consequently
$$P_n(x) = \sum_{k=0}^n (-1)^k  \Big(\sum_i \binom{i}{k} \xi_i \Big) t^k$$
$$P'_n(x) = \sum_{k=1}^n (-1)^{k-1}  \Big(\sum_i i \binom{i-1}{k-1} \xi_i \Big) t^{k-1} = \sum_{k=1}^n (-1)^{k-1}  k\Big(\sum_i i \binom{i}{k} \xi_i \Big) t^{k-1} $$
%{\bf I allow $t=0$ here in order to obtain the new statement of Theorem 2.3. Please check.} 
 
Notice easily that, with probability $1-\exp(-\Omega(\log^2 n))$, 
\begin{equation}\label{eqn:xlarge:expansion:1,2}
|\sum_i \binom{i}{k} \xi_i | = n^{k+1/2} \log^{O(1)} n,  \forall 1\le k\le n.
\end{equation}

%{\color{blue} To Hoi: It seems that  second estimate is not true (say $k=n$ the second bound could be as large as $n^{k+3/2}(\log n)^{O(1)}$). If we correct it to $n^{k+3/2}(\log n)^{O(1)}$ then it follows from the first estimate, so is not needed. In fact I think the first estimate would be enough in the argument below, more specifically to get to \eqref{eqn:near1:3} from  \eqref{eqn:near1:1} and \eqref{eqn:near1:2}.}

Let $\CE$ denote this event, on which we will condition for the rest of our argument. Thus
\begin{equation}\label{eqn:xlarge:error}
\sum_{k\ge k_0} (-1)^k (\sum_{0\le i\le n} \binom{i}{k} \xi_i) t^k  = (n t)^{k_0} n^{1/2} \log^{O(1)}n.
\end{equation}

In particular, with $k_0=1$, 
$$\sum_{k\ge 1} (-1)^k (\sum_{0\le i\le n} \binom{i}{k} \xi_i) t^k  = O(n^{-1/2+\ep} \log^{O(1)}n), $$
therefore $\sum_i \xi_i = P_n(x)+o(1)$. But as $\xi_i$ takes integer values and $|P_n(x)| = o(1)$, therefore by taking $n$ large we must have
\begin{equation}\label{eqn:near1:0'}
\sum_{i=0}^n \xi_i=0.
\end{equation}

%{\color{blue}To Hoi: by definition of type I, $\xi$ takes value in $\{\pm 1, \dots, \pm N\}$.}

After replacing \eqref{eqn:near1:0'} into $|P_n(x)|$, we obtain
\begin{equation}\label{eqn:near1:1}
t |\sum_{k\ge 1} (-1)^k (\sum_{0\le i\le n} \binom{i}{k} \xi_i) t^{k-1}| = |P_n(x)| = O(n^{-B}).
\end{equation}

Now we consider the assumption that $|P'_n(x)|\le n^{-B}$, from which we infer that
\begin{equation}\label{eqn:near1:2}
 |P'_n(x)| =| \sum_{k\ge 1} (-1)^k k(\sum_{0\le i\le n} \binom{i}{k} \xi_i) t^{k-1}|= 
\end{equation}
$$=|-\sum_{0\le i\le n} i \xi_i +  \sum_{k\ge 2} (-1)^k k(\sum_{0\le i\le n} \binom{i}{k} \xi_i) t^{k-1}| = O(n^{-B}).$$
We next consider two cases.

{\bf Case 1. $t\le n^{-8}$}. It follows easily from \eqref{eqn:xlarge:expansion:1,2} and \eqref{eqn:near1:2} that 
$$\sum_{0\le i\le n} i \xi_i =  o(1) \ \ .$$

{\bf Case 2. $t\ge n^{-8}$}. As $B\ge 16$, it follows from \eqref{eqn:near1:1} that 
\begin{equation}\label{eqn:near1:3}
|\sum_{k\ge 1} (-1)^k (\sum_{0\le i\le n} \binom{i}{k} \xi_i) t^{k-1}| = | -\sum_{0\le i\le n} i \xi_i + \sum_{k\ge 2} (-1)^k (\sum_{0\le i\le n} \binom{i}{k} \xi_i) t^{k-1}| \le n^{-B/2}.
\end{equation}

Combining \eqref{eqn:near1:2} and \eqref{eqn:near1:3} together to eliminate the term corresponding to $k=2$, and also by \eqref{eqn:xlarge:expansion:1,2}
\begin{equation}\label{eqn:near1:3}
\sum_{0\le i \le n} \binom{i}{1}\xi_i +O(n^{7/2}\log^{O(1)}n) t^2= O(n^{-B/2}).
\end{equation}

Thus, using the fact that $t\le n^{-2+\epsilon}\le n^{-15/8}$ with $\ep\le 1/8$, it follows that 
$$|\sum_i i \xi_i|=o(1).$$

As $\xi_i$ takes integer values, it follows from both cases that 
\begin{equation}\label{eqn:near1:4}
\sum_i i\xi_i =0.
\end{equation}

%{\color{blue} To anh Van: we corrected the estimates here.}

\section{Application: proof of Theorem \ref{theorem:expectedvalue:2}}\label{section:expectedvalue:2}
%{\color{blue} To Yen: we need to rewrite this part to cover both Type I and II.}

%Our treatment is similar to \cite{NgOV} with some modifications.
 
%As it is well known (see for instance \cite{Book1},\cite{Book2}) that
Since $\xi$ has no atom at $0$, $P_n(0) \ne 0$ with probability $1$.  First of all, the contribution towards the expectation at  the points $\pm 1$ are negligible owing to the following elementary estimates.

\begin{claim}\label{l.endpoint}We have
\begin{align*}
&\P\Big(P_n(1)P_n(-1)=0\Big) = O(1/\sqrt n)\\
&\P\Big(P_n(1)=0 \wedge P_n'(1)=0\Big) = O(1/n^2)\\
&\P\Big(P_n(-1)=0 \wedge P'_n(-1)=0\Big)  = O(1/n^2).
\end{align*}
\end{claim}

Note that the above claim is trivial in the Type II setting, so it remains to show these estimates for Type I. The first estimate clearly follows from the classical Erd\H{o}s-Littlewood-Offord bound. We refer the reader to Claim \ref{claim:anticoncentration} for a short proof of the remaining two estimates.

%{\color{red} To Yen: I changed a bit here. Also, can you rewrite the multi equations in align* mod for consistency? }

Using Claim~\ref{l.endpoint}, it follows that
\begin{align*}
\E N_{n,\xi} \{-1,1\}  &= O(1/\sqrt n) + n O(1/n^2) \\ 
&= O(1/\sqrt n) \ \ .
\end{align*}

Now, using the fact that the (real) zero sets of $P_n(x)$ and $x^n P_n(1/x)$ have the same distribution, it follows that 
\begin{align*}
\E N_{n,\xi} &= 2 \E N_{n,\xi}(-1,1) + \E N_{n,\xi} \{-1,1\} \\
&= 2 \E N_{n,\xi} (-1,1) + O(1/\sqrt n) \ \ .
\end{align*}

%{\color{blue}To anh Van: We will need to separate the contributions of the endpoints $-1$ and $1$ anyway since the interval $(1,\infty)$ and $(0,1]$ are not isomorphic via the map $x\to 1/x$ (the endpoint $1$ is left out). Also this means we don't have to worry about the endpoints for the rest of the proof.}

%{\bf Just double check: Do we  need to assume $\xi$ has symmetric distribution to have this ? } 

In the following, we will consider the number of real roots in $(0,1)$, and show that
\begin{equation}\label{eqn:(0,1)}
\E N_{n,\xi}(0,1) = \frac 1 {2\pi} \log n + C + o(1)
\end{equation}
for some $C=C(\xi)$. 

To estimate $\E N_{n,\xi}(-1,0)$, we consider the random polynomial
$$\widetilde P_n(x) := P_n(-x) = \xi_0 -\xi_1 x+ \xi_2x^2 +\dots + (-1)^n \xi_n x^n$$ 
and let $\widetilde N_{n,\xi}$ denote its number of real zeros. By definition, we have
$$\E N_{n,\xi}(-1,0) = \E \widetilde N_{n,\xi}(0,1)$$
so we need to estimate the average number of real zeros for $\widetilde P_n$ in $(0,1)$. 

Now, Type I distributions are symmetric so in that setting the zero sets of $\widetilde P_n$ and $P_n$ would have the same distribution, therefore $\E \widetilde N_{n,\xi} = \E N_{n,\xi}$ and one obtains the same estimate as above for $\E N_{n,\xi}(-1,0)$. 

For Type II distributions, the argument we use below for estimating $\E N_{n,\xi}(0,1)$ could be applied to estimate $\E \widetilde N_{n,\xi}(0,1)$. Most importantly,  our result on non-existence of double roots (Theorem \ref{theorem:repulsion:general'}) does not require the coefficients $\xi_0,\dots, \xi_n$ to have identical distributions, and could be applied to $\widetilde P_n$. Also the cited ingredients that we used below (Theorem~\ref{theorem:TV} and Lemma~\ref{lemma:Jensen}) do not require $\xi_0,\dots, \xi_n$ to have identical distributions.  With cosmetic changes, one could use the argument below to $\widetilde P_n$ and obtains a similar estimate for $\E \widetilde N_{n,\xi}(0,1)$, which is the same as $\E N_{n,\xi}(-1,0)$, and conclude the proof of Theorem~\ref{theorem:expectedvalue:2}. For the rest of the section we will be focusing on \eqref{eqn:(0,1)}.

\subsection{Comparison lemmas} 

Our first tool is the following result from \cite{TVpolynomial}. 

\begin{theorem}  \label{theorem:TV} There is a positive constant $\alpha$ such that the following holds. 
Let $\ep >0$ be an arbitrary small constant and  $\xi_0,\dots,\xi_n$ be independent random variables with mean 0, variance 1 and uniformly bounded $(2+\ep)$-moment. 
There is a constant $C_1= C_1(\ep) $ such that for any  $n \ge  C_1$ and interval $I := (1-r, a) \subset (1- n^{-\ep} , 1]$
\begin{equation}  \label{small2} 
\E N_{n, \xi_0,\dots,\xi_n} I   = \E N_{n, N(0,1) } I   + O(r ^{\alpha} ),
\end{equation} 

where the implicit constant in $O(.)$ depends only on $\alpha$ and $\ep$.
\end{theorem} 

%{\color{blue} To anh Van: we restated with non iid $\xi_i$ here.}

Next, for convenience, we  truncate the random variables $\xi_0,\dots, \xi_n$. Let $d>0$ be a parameter and let $\CB_d$ be the event $|\xi_0|<n^d \wedge \dots \wedge |\xi_n|<n^d$. As $\xi$ has mean zero and bounded $(2+\ep_0)$-moment, we have the following elementary bound for $d\ge 1$
$$\P(\CB_{2d}^c)=O( n^{1-3d}).$$

In what follows we will condition on $\CB_{4}$. Consider $P_{n}(x) = \sum_{i=0} ^n \xi_i x^i$ and  for $m <n$, we set 
$$g_{m} := P_n -P_m =\sum_{i=m+1} ^n \xi_i x^i.$$  

For any $0< x \le 1-r$, a generous Chernoff's bound yields that  for any $\lambda >0$
\begin{align*}
&\P \Big(|g_m(x) | \ge (\lambda+1) n^5 \sqrt  { \sum_{i=m +1}^n (1-r)^{2i} }\bigg| \CB_4 \Big ) \quad \le \\
\le \quad &\P \Big (|g_m(x)|  \ge (\lambda+1) n^5 \sqrt  { \sum_{i=m}^n x^{2i} }\bigg| \CB_4\Big) \quad \le \\
\le \quad &2 \exp( -\lambda^2/2 )  \ \ . 
\end{align*}

%{\bf If we do not assume symmetry, after conditioning, the $\xi_i$ may not have mean zero, does it matter wrt to  the large deviation bound ?}  

%{\color{blue} To anh Van: we added $(\lambda+1)n^5$ here to compensate the loss of mean zero.}

Since  $$\sum_{i=m+1}^n (1-r)^{2i} \le (1-r)^{2m+2 } \frac{1}{1 -(1-r)^2 } := s(r,m),  $$ it follows that 
%\begin{equation}\label{eqn:gT}
%\P (|g_m| \ge  \lambda n^2 s(r,m)|\CB_2 ) \le 2 \exp (-\lambda^2/2 ). 
%\end{equation}
\begin{equation}\label{eqn:gT}
\P (|g_m| \ge  (\lambda+1) n^4 \sqrt{s(r,m)}|\CB_4 ) \le 2 \exp (-\lambda^2/2 ). 
\end{equation}

We next compare the roots of 
$P_n$ and $P_m$ in the interval $(0, 1-r)$. 

\begin{lemma}[Roots comparison for truncated polynomials]\label{lemma:comparison} Let $r \in (1/n, 1)$ and $m \le n$ such that  $m \ge 4B r^{-1} \log n$, where $B=B(3)$ of Theorem \ref{theorem:repulsion:general}. Then for any  subinterval $J$ of $ (0,1-r]$ one has
\begin{equation} \label{Tn3} | \E N_{n,\xi} J  - \E N_{m,\xi} J |   =   O(  m^{-2} )  \ \ , 
\end{equation} 
and the implicit constant depends on $N$ and $B$ only.
\end{lemma}

To prove Lemma \ref{lemma:comparison}, we need the following elementary lemma from \cite{NgOV}.

\begin{lemma}  \label{approximation} 
 Assume that $F (x) \in C^2(\R)$ and $G(x)$ are continuous functions satisfying the following properties
 \begin{itemize}
 \item $F(x_0) =0$ and $|F'(x_0)| \ge \epsilon_1$;
 \vskip .1in
 \item  $|F^{''} (x)  | \le M$ for all $x \in I := [x_0 -\epsilon_1 M^{-1} , x_0 +\epsilon_1 M^{-1} ]$; 
 \vskip .1in
 \item $\sup_{x \in I} | F(x)- G(x) | \le \frac{1}{4} \epsilon_1^{2} M^{-1}$. 
 \end{itemize}
 Then $G$ has a root in $I$.
\end{lemma}

\begin{proof}[Proof of Lemma \ref{lemma:comparison}]  Conditioned on $\CB_4$, with probability at least $1 - 2\exp( -\log^2 n/2 ) \ge 1 - n^{-\omega (1) }$ the following holds
 $$|P_n (x)- P_m(x) | \le n^5(\log n +1)  \sqrt{s(r, m)} = n^5(\log n +1) (1-r)^{m+ 1} \frac{1}{\sqrt{1 -(1-r)^2} } \le n^{-3B}  $$  
for all $0 \le x \le 1-r$ and $n$ sufficiently large.

By  Theorem \ref{theorem:repulsion:general} or Theorem \ref{theorem:repulsion:general'} (with $C=3$), $|P_n '(x) | \ge n^{-B}$ for all $x\in J$ with probability $1-O( n^{-3})$. Note that the conditioning on $B_4$, which holds with probability at least $1-O(n^{-5})$, will not affect the $O(n^{-3})$ term in this estimate. Applying Lemma \ref{approximation} with $\epsilon_1 =n^{-B}, M=  n^3$, $F= P_n, G = P_m$, we conclude that with probability $1-O(n^{-3})$, for any root $x_0$ of $P_n(x)$ in the interval $(0, 1-r)$  (which is a subset of $(0, 1-1/n)$), there is a root $y_0$ of $P_m (x)$ such that $|x_0 -y_0| \le \epsilon_1 M^{-1}=  n^{-B-3}$.

On the other hand, applying (2) of Theorem \ref{theorem:repulsion:general} or Theorem \ref{theorem:repulsion:general'} with $C=3$, again with probability $1 -O(n^{-3})$ there is no pair of roots of $P_n$ in $J$ with distance less than $n^{-B}$. It follows that  for different roots $x_0$ we can  choose different roots $y_0$. Furthermore, by (3) of Theorem \ref{theorem:repulsion:general} or Theorem \ref{theorem:repulsion:general'}, with probability $1-O(n^{-3})$, all roots of 
$P_n(x)$ must be of distance at least $n^{-B} $ from the two ends of the interval.  If this holds, then all $y_0$ must also be inside the interval. 
This implies that with probability at least $1 -O(n^{-3})$,   the number of roots of $P_m$ in $J$  is at least that of $P_n$. Putting together, we obtain 
\begin{equation} \label{Tn1} \E N_{m,\xi} J  \ge \E N_{n,\xi} J -  O(n^{-3})  n  \ge \E N_{n,\xi} J - O( n^{-2}), 
\end{equation} 
where the factor $n$  comes from the fact that $P_n$ has at most $n$ real roots.

Switching the roles of $P_n$ and $P_m$, noting that as $r \ge 4B\log n/m > 1/m$, 
$$J \subset (0, 1-r] \subset (0, 1-1/m).$$ 

As such, Theorem \ref{theorem:repulsion:general} and Theorem \ref{theorem:repulsion:general'} are also applicable to $P_m(x)$. Argue similarly as above, we also have 
\begin{equation} \label{Tn2} \E N_{n,\xi} J  \ge \E N_{m,\xi} J -  (O(m^{-3})+n^{-3B}) m  \ge  \E N_{m,\xi} J - O( m^{-2}). 
\end{equation} 

It follows that 
$$| \E N_{n,\xi} J  - \E N_{m,\xi} J |   =  O(m^{-2}) .$$
\end{proof}

\subsection{Control of the error term} By iterating Lemma \ref{lemma:comparison}, one can achieve the following.

\begin{corollary}\label{cor:comparison:L} Let $C_0$ be sufficiently large, and let  $I= (0, 1 - C_0^{-1} )$. Then for any sufficiently large integer $L$ (depending on $B=B(3)$ and $C_0$) and $n\ge L$
$$\Big | \E N_{n,\xi} I- \E N_{L,\xi} I \Big| =O(L^{-1}) \ \ , $$
where the implicit constant depends only on $B(3)$ and the parameter $N$ of $\xi$.
\end{corollary}

\begin{proof}[Proof of Corollary \ref{cor:comparison:L}] Let $r=1/C_0$. We assume $L > C_0$ so that $r\in (1/n,1)$ for every $n\ge L$. Define the sequence $\{n_i\}$ with $n_0=n$ and  $n_{i+1}= 1+\lfloor 4B r^{-1} \log n_i \rfloor$.  By ensuring that $C_0$ is sufficiently large we will have $n_{i}>n_{i+1}$ unless $n_{i}=1$. Thus, the sequence $\{n_i\}_{0}^k$ is decreasing, and let $k$ be the first index where $n_{k+1}\le L$. Then
\begin{align*}
| \E N_{n,\xi} I- \E N_{L,\xi} I | &\le  \sum_{i=0}^{k-1}| \E N_{n_i,\xi} I- \E N_{n_{i+1},\xi} I |+ | \E N_{n_k,\xi} I- \E N_{L,\xi} I |\\ 
&= O(\sum_{i=0}^{k-1} n_{i+1}^{-2}+ L^{-2})=O(L^{-1}).
\end{align*}
 Here the last estimate also holds by applying Lemma~\ref{lemma:comparison} for $n=n_k$ and $m=L$, and clearly $L\ge n_{k+1} > 4Br^{-1}\log n_k$, and $r\in (1/L,1) \subset(1/n_k,1)$.
\end{proof}

Next, we use 

\begin{lemma}\cite[Lemma 2, Remark 4]{NgOV}\label{lemma:Jensen} Assume that $\xi_0,\dots,\xi_n$ have mean 0, variance 1, and uniformly bounded $(2+\ep_0)$-moments. Then there exist constants $N_1, N_2$ such that
$$ \E N_{n, \xi} [0, 1- C_0^{-1} )  \le \frac{1}{2\pi}   \log C_0 +N_2,$$

provided that $C_0\ge N_1$,
\end{lemma}

%{\color{blue} To anh Van: we stated here for non iid $\xi_i$.} 

For each $L$, denote $C_{L} :=   \E N_{L,\xi} I = \E N_{L,\xi} [0,  1- C_0^{-1}]$.  By Lemma \ref{lemma:Jensen}, $C_L$ are uniformly bounded provided that $C_0\ge N_1$, thus there is a subsequence of $C_{L}$ which tends to a finite limit $C^\ast=C^\ast(C_0)$.  By applying Corollary \ref{cor:comparison:L}, we obtain the following.

\begin{corollary}\label{cor:limit} For any $C_0>0$, there exists $C^\ast=C^\ast(C_0)<\infty$ such that
$$ \lim_{n\rightarrow \infty} \E N_{n,\xi} (0,1-C_0^{-1})= C^\ast .$$ 
%{\color{blue}Furthermore, for any $L$ sufficiently large (depending on $C_0$ and $B(3)$), we have the following estimates for every $n\ge L$:
%$$ |\E N_{n,\xi}(0,1-C_0^{-1}) - C^\ast|  = O_{B(3),N}(1 / L)$$
%}
\end{corollary}

\subsection{Control of the main term} We next turn to the main term by utilizing Theorem \ref{theorem:TV} and Lemma \ref{lemma:comparison}. Recall the constant $\alpha$ from Theorem~\ref{theorem:TV}. Let $0<\ep<\alpha/2$ be a small constant (so that in particular $\ep<1$), and let $C_1=C_1(\ep)$ to be the constant of Theorem \ref{theorem:TV}.

%{\color{blue}To Hoi: The proof below does not include the endpoint $1$ anyway, so I changed the statement of the Lemma~\ref{lemma:mainterm}. In fact you can never get the end point $1$ from Theorem~\ref{theorem:TV} since the interval $I$ as stated in this theorem has to be open. Fortunately we don't need the endpoint here anyway, instead we need Lemma~\ref{l.endpoint}.}

\begin{lemma}\label{lemma:mainterm} 
Assume that $C_0 > C_1^{\ep}$, then  
$$| \E N_{n,\xi} (1- {C_0}^{-1} ,1) - \E N_{n, N(0,1)}  (1- {C_0}^{-1} ,1)| = O( {C_0} ^{-\alpha}),$$
here the implicit constant depends only on $N$, $\alpha$, and $\ep$.
\end{lemma}

%{\color{blue} To Hoi: The proof below does not seem to give $O(C_0^{-1})$. It gives a worse bound $O(C_0^{-\alpha})$, but that will probably be enough. And we don't seem to need $\epsilon < \alpha/2$.}

We will show the following equivalent statement: assume that $C_0>C_1$, then 
\begin{equation}\label{eqn:mainterm:side}
| \E N_{n,\xi} (1- {C_0}^{-\ep} ,1) - \E N_{n, N(0,1)}  (1- {C_0}^{-\ep} ,1) | =O( {C_0} ^{-\ep \alpha}).
\end{equation}

\begin{proof}[Proof of Lemma \ref{lemma:mainterm}]  We will justify \eqref{eqn:mainterm:side} following \cite{NgOV}. Set $n_0:=n, r_0 = n^{-\ep}$ and define recursively 
$$ n_{i} := 1+ \lfloor 4B  r_{i-1}^{-1}  \log n_{i-1} \rfloor , \mbox{ and } r_{i} := n_{i}^{-\ep}, i\ge 1.$$ 

It is clear that $\{n_i\}$ and $\{r_i\}$ are, respectively, strictly decreasing and increasing sequences . Let $L$ be the largest index such that $n_L > C_0$. It follows that $C_0 \ge n_{L+1} > 4B n_L^{\ep} \log n_L\ge n_L^{\ep}$, therefore $n_L <  C_0^{1/ \epsilon}$. It also follows that
$$4B\log n_L < C_0^{1-\epsilon}$$
Redefine $n_{L+1}:=1+ \lfloor C_0 \rfloor$ and $r_{L+1}=C_0^{-\ep}$. %=C_0^{-\ep}$. 
It is clear that we still have $n_L \ge n_{L+1} \ge 4B r_L^{-1} \log n_L$.

For $1\le i\le L+1$, define $I_i := (1-r_{i}, 1-r_{i-1}]$. For every $1\le j\le i$ we have  $I_i \subset (0,1-r_{j-1}]$ while $r_{j-1}\in (1/n_{j-1}, 1)$. Thus, by \eqref{Tn3}, for any $1\le j\le i$ we have
$$|\E N_{n_{j-1},\xi} I_i  -\E N_{n_{j},\xi } I_i  |  = O( n_{j}^{-2} ). $$

By the triangle inequality,
\begin{equation}\label{eqn:triangle:1}
|\E N_{n_0,\xi} I_i  -\E N_{n_{i},\xi } I_i  |  = O( \sum_{j=1}^{i} n_{j}^{-2}) =O( n_{i}^{-1}). 
\end{equation}

Similarly, as standard Gaussian distribution is of type II, 
\begin{equation}\label{eqn:triangle:Gaussian}
|\E N_{n_0,N(0,1)} I_i  -\E N_{n_{i},N(0,1) } I_i  |  = O( \sum_{j=1}^{i} n_{j}^{-2}) =O( n_{i}^{-1}). 
\end{equation}

On the other hand, note that $n_i \ge C_1$ for $i\le L+1$, and $I_i=(1-n_{i}^{-\ep}, 1-n_{i-1}^{-\ep}]\subset (1-n_i^{-\ep},1)$.  By Theorem \ref{theorem:TV} 
\begin{equation}\label{eqn:triangle:2}
|\E N_{n_{i},\xi } I_i - \E N_{n_{i} , N(0,1)} I_i  | \le  O(n_i^{-\epsilon \alpha} ). 
\end{equation}

Combining \eqref{eqn:triangle:1} and \eqref{eqn:triangle:2}, one obtains
\begin{equation}\label{eqn:triangle:3}
|\E N_{n_0,\xi} I_i  -\E N_{n_0,N(0,1) } I_i  |  =O(n_{i}^{-1} +  n_i^{-\epsilon \alpha}). 
\end{equation}

Let $I =\cup_{i=0}^{L+1} I_i$, again by the triangle inequality
$$|\E N_{n,\xi} I - \E N_{n, N(0,1)} I |   =O( \sum_{i=0}^{L+1} n_i^{-1} +  \sum_{i=0}^{L+1}    n_i^{-\epsilon \alpha}). $$

The right end point of $I$ is $1-n^{-\ep}$, and the left end point is $1 -r_{L+1} = 1 -C_0^{-\ep}$. Furthermore, by definition of the $n_i$, it is easy to see that $(n_i)_{i=0}^{L+1}$ is lacunary, therefore
$$  \sum_{i=0}^{L+1} n_i^{-\epsilon \alpha } =O( n_{L+1}^{-\epsilon \alpha})  =O( {C_0} ^{-\ep \alpha} ) \ \ .$$

Thus, 
$$| \E N_{n,\xi} (1 -C_0^{-\ep},1-n^{-\ep} ]) - \E N_{n, N(0,1)} (1 -C_0^{-\ep},1-n^{-\ep} ] | = | \E N_{n,\xi} I - \E N_{n, N(0,1)} I | =  O( {C_0}^{-\ep \alpha}).$$ 

Combined with Theorem~\ref{theorem:TV}, 
$$| \E N_{n,\xi} (1-C_0^{-1}, 1)- \E N_{n, N(0,1)} (1-C_0^{-1}, 1) | = O( {C_0}^{-\ep \alpha}) + O(n^{-\epsilon\alpha}) = O( {C_0}^{-\ep \alpha}),$$

proving \eqref{eqn:mainterm:side}.

%{\color{blue} To anh Van: The above interval does not include the endpoint $\{1\}$. Does  Theorem~\ref{theorem:TV} cover half-closed intervals $(�,1]$?}

%{\bf Yes. I think so } 

%{\color{blue} To Hoi:  have you asked Van about the endpoint $\{1\}$?}
\end{proof}

\subsection{Completing the proof of Theorem \ref{theorem:expectedvalue:2}} It suffices to show that 
\begin{equation}\label{eqn:EK:uniform}
\E N_{n,\xi}(0,1)=\frac{1}{2\pi}\log n + C_\xi+o(1).
\end{equation}

To this end, we first complement the result of Lemma \ref{lemma:mainterm} by giving an estimate for $\E N_{n,N(0,1)}(1-C_0^{-1},1)$.

%{\color{blue} To Hoi (Updated): The original claim $\E  N_{n, N(0,1) }  (1- C_0^{-1} , 1] =\frac{1}{2\pi}\log n + O(C_0^{-1})$ is not true in general. For instance take $C_0=n$, the left hand side is bounded above by $1$: the integrand in the Edelman-Kostlan is bounded by $n$ (this is not hard to check, heuristically this means the multiplicity of every root is at most $n$) and the interval of integration is $(1-1/n,1]$. On the other hand the right hand side would be as large as $\log n/(2\pi)$ plus $O(1/n)$!. 
% I have corrected this claim, which will be enough for the application in this paper.} 

\begin{claim}\label{claim:gaussian} For every $C_0\in (1,\infty)$ there exists a finite number $B^\ast=B^\ast(C_0)$ that depends only on $C_0$ such that 
$$\E N_{n, N(0,1) }  (1- C_0^{-1} , 1) =\frac{1}{2\pi}\log n + B^\ast + o_{C_0}(1) \ \ ,$$
here the $o_{C_0}(1)$ term is with respect to the limit $n\to\infty$, and the implied constant depends on $C_0$.
\end{claim}

\begin{proof}[Proof of Claim \ref{claim:gaussian}]

We will use the following formula from \cite{EK}, which  asserts that the following equality holds for every interval $I \subset [0,1]$:
$$\E N_{n, N(0,1) } I = \int_{I}  \frac{1}{\pi} \sqrt { \frac{1}{(t^2-1)^2} - \frac{(n+1)^2 t^{2n} }{ (t^{2n+2} -1)^2 } } dt.$$
In particular, for every fixed $1<C_0<\infty$ we have
$$\lim_{n\to\infty} \E N_{n,N(0,1)} [0,1-C_0^{-1})  = \int_0^{1-C_0^{-1}} \frac{1}{\pi(1-t^2)}dt$$
thus we can define
$$B^\ast(C_0) = \frac 1 4C_{Gau} - \int_0^{1-C_0^{-1}} \frac{1}{\pi(1-t^2)}dt $$
here recall that $C_{Gau}$ is the constant in the asymptotics expansion  \eqref{eqn:EK}  of the Gaussian Kac polynomial.
%$$\E N_{n,N(0,1)} = \frac 2 \pi \log n + C_{Gau} + O(1/n)$$ 
\end{proof}

%{\color{blue} To Yen: can you help us finishing the claim? This formula was used by Edelman and Costlan to deduce \eqref{eqn:EK}. For our application, any error bound such as $O(C_0^{-.1})$ would be equally useful. I can also ask Oanh to help if you don't have time. }

It then follows from Lemma \ref{lemma:mainterm} that 
$$\E N_{n, \xi}  (1- C_0^{-1} , 1) =\frac{1}{2\pi}\log n + B^\ast(C_0) + O(C_0^{-\alpha}) + o_{C_0}(1) \ \ ,$$
and in the $O(C_0^{-\alpha})$ term the implicit constant may depend on $\alpha,\epsilon$ and $\xi$.

Combining with Corollary \ref{cor:limit}, we obtain
$$ \E N_{n, \xi}(0,1)    = \frac{1}{2\pi} \log n + B^\ast(C_0) + C^\ast (C_0) +   O(C_0^{-\alpha}) +o_{C_0}(1), $$ 
where $C^\ast = C^\ast(C_0)$ is a number depending on $C_0$.

Replacing $C_0$ by $C_0'$ and subtract,
$$|B^\ast(C_0') + C^\ast(C_0') -B^\ast(C_0) - C^\ast(C_0) | =   O(C_0^{-\alpha} + C_0'^{-\alpha}) + o_{C_0,C_0'}(1),$$
and sending $n\to\infty$ we obtain
 $$|B^\ast(C_0') + C^\ast(C_0') -B^\ast(C_0) - C^\ast(C_0) | =   O(C_0^{-\alpha} + C_0'^{-\alpha}) $$

which shows that the function $B^\ast(C_0)+ C^\ast(C_0)$ tends to a limit  as $C_0$ tends to infinity.  Denote this limit by $C_{\xi}$, it follows that
$$\E N_{n,\xi}(0,1)  =\frac{1}{2\pi} \log n +  C_{\xi} + o(1), $$ 

concluding the proof.

\begin{remark}
Note that in this application section, we have applied our roots repulsion results (Theorem \ref{theorem:repulsion:general} and Theorem \ref{theorem:repulsion:general'}) only for the interval $(0,1-\log^2n/n]$. 
\end{remark}

%{\color{blue} To Yen: I rewrote a little bit, please check and correct accordingly whenever necessary.}

\appendix

\section{Sharpness of Fact \ref{fact:arithmetic} and proof of Claim \ref{l.endpoint}} \label{appendix:1} 

We first address Fact \ref{fact:arithmetic}. It is clear that when $n$ is even then $P_n(1),P_n(-1)$ are odd numbers, and hence the polynomial cannot have double roots at these points. We next show the same thing for the case $n=4k+1$.

\begin{fact}\label{fact:4k+1}
Assume that $n=4k+1$, then $P_n(x)$ cannot have double root at $-1$ or $1$.
\end{fact}

\begin{proof}(of Fact \ref{fact:4k+1})
Assume that $P_n(1)=P_n'(1)=0$, then one has
$$\xi_0+\xi_1+\dots+\xi_{4k}+\xi_{4k+1}= 0 \mbox{ and } \xi_1+2\xi_2 +\dots + 4k\xi_{4k}+(4k+1)\xi_{4k+1}=0.$$

Consequently, 
$$\xi_2 +2\xi_3 + \dots + (4k-1)\xi_{4k}+ 4k \xi_{4k+1}=\xi_0.$$

However, this is impossible as the RHS is an odd number, while the LHS is clearly an even number for any choice of $\xi_2,\dots,\xi_{4k+1} \in {\pm 1} $. 

The non-existence of double root at $-1$ can be argued similarly (or just by setting $Q_n(x):=P_n(-x)$.)
\end{proof}

%{\color{blue} To anh Van: we added a short fact for $n=4k+1$ here.}

We now give a brief explanation that the probability bound of Fact \ref{fact:arithmetic} is sharp, up to a multiplicative constant. 

\begin{lemma}\label{lemma:tight} Let $\xi$ be a Bernoulli random variable and $n+1$ be divisible by 4. Then 
$$ \P( P_n(1) = P'_n (1) =0) =\Omega (n^{-2}) . $$
\end{lemma}

We use the circle method. Let $v_i:=(1,i)$, and $V=\{v_0,\dots,v_n\}$. Let $p\gg n$ be a sufficiently large prime. We first write
\begin{align}\label{eqn:tight:1}
\P(\sum_i \xi_i v_i=0) &=\E_{\xi_0,\dots,\xi_n}\E_{\Bx\in(\Z/p\Z)^2} e_p (\langle \sum_{i=1}^n \xi_i v_i,\Bx \rangle)=\E_{\Bx\in(\Z/p\Z)^2}\prod_{i=0}^n \cos(2\pi \langle v_i,\Bx \rangle /p).
\end{align}

Observe that $|\cos(\pi 2wx/p)| \le 3/4 + \cos(4\pi wx/p)/4 \le \exp(-\|2w x /p\|^2)$, where $\|.\|$ is the distance to the nearest integer. We are going to analyze the sum 
$$\sum_{v_i\in V} \| 2\langle \Bx, v_i \rangle/p\|^2 = \sum_{i=0}^n \|(2x_1+2i x_2)/p\|^2 .$$ 

Basically, if this sum is quite large, then its distribution in \eqref{eqn:tight:1} is negligible. We state the following elementary claims whose proofs are left to the reader as an exercise.

\begin{claim}\label{claim:integral} The following holds.  
\begin{itemize}
\item Assume that $\|\frac{2x_2}{p}\|\ge  \frac{\log^2n}{n^{3/2}}$, then 
$$ \sum_{i=0}^n \|(2x_1+2i x_2)/p\|^2 \ge 4\log n.$$
\vskip .1in
\item Assume that $\|\frac{2x_2}{p}\| \le  \frac{\log^2n}{n^{3/2}}$ and $\|\frac{2x_1}{p}\|\ge  \frac{\log^2n}{n^{1/2}}$, then 
$$\sum_{i=0}^n \|(2x_1+2i x_2)/p\|^2 \ge 4\log n.$$
\end{itemize}
\end{claim}

It follows from Claim \ref{claim:integral} that the main term of the sum in \eqref{eqn:tight:1} is governed by  $\|2x_1/p\|\le  \log^2n/n^{1/2}$ and $\|2x_2/p\| \le  \log^2n/n^{3/2}$. Thus either $\|x_1/p\|\le  \log^2n/2n^{1/2}$ or $\|x_1/p + 1/2\|\le  \log^2n/2n^{1/2}$ and $\|x_2/p\| \le  \log^2n/2n^{3/2}$ or $\|x_2/p+1/2\| \le  \log^2n/2n^{3/2}$. As $4|n+1$, the interested reader is invited to check that the contribution in \eqref{eqn:tight:1} of one of these four cases  are the same. It is thus enough to work with the case $\|x_1/p\|\le  \log^2n/2n^{1/2}$ and   $\|x_2/p\| \le  \log^2n/2n^{3/2}$. We are going to show the following.

\begin{lemma}  \label{lemma:S1}
\begin{align*}
S:& =  \frac{1}{p^2} \sum_{\|\frac{x_1}{p}\| \le \frac{\log^2 n}{n^{1/2}}, \|\frac{x_2}{p}\| \le \frac{\log^2 n}{n^{3/2}} } \prod_{i=0}^n   \cos \frac{2\pi  (x_1+ ix_2)}{p} \\
&= \frac{1}{p^2} \sum_{\|\frac{x_1}{p}\| \le \frac{\log^2 n}{n^{1/2}}, \|\frac{x_2}{p}\| \le \frac{\log^2 n}{n^{3/2}} } \prod_{i=0}^n   \Big|\cos \frac{2\pi  (x_1+ ix_2)}{p}\Big|\\ 
&= \Omega (n^{-2}).
\end{align*}
\end{lemma}

The method below gives an exact estimate on the asymptotic constant, but we will not need this fact.

\begin{proof}[Proof of Lemma \ref{lemma:S1}] The first equality is trivial, as all cosines are positive in this range of $\Bx$. Viewing $x_1$ and $x_2$ as integers with absolute value at most $p\log^2 n  /n^{1/2}$ and   $ p\log^2  n  /n^{3/2}$ respectively. We have
$$\cos \frac{2\pi  (x_1 + i x_2)}{p} = 1 -\Big(\frac{1}{2} +o(1)\Big) \frac{4\pi^2 (x_1+ix_2)^2 }{p^2}   =\exp\Big(- (1/2+o(1)) \frac{4\pi^2 (x_1+ix_2)^2 }{p^2 }\Big) . $$

It follows that as $p\rightarrow \infty$

\begin{align*}
S &=\Big(1+o(1)\Big) \int_{|x_1|\le \frac{\log^ 2 n}{n^{1/2}}, |x_2| \le \frac{\log^ 2 n}{n^{3/2}} } \exp\Big(- (1/2+o(1)) 4\pi^2 \sum_{i=1}^n(x_1+ix_2)^2   \Big) dx_1dx_2\\
&\ge  \Big(1+o(1)\Big) \int_{|x_1|\le \frac{\log^ 2 n}{n^{1/2}}, |x_2| \le \frac{\log^ 2 n}{n^{3/2}} } \exp\Big(- (1/2+o(1))  4\pi^2 \sum_{i=1}^n2(x_1^2+i^2x_2^2) \Big) dx_1dx_2\\
&\ge \Big(1+o(1)\Big) \int_{|x_1|\le \frac{\log^ 2 n}{n^{1/2}}} \exp\Big(- (1/2+o(1))  8\pi^2 nx_1^2\Big) dx_1 \int_{|x_2| \le \frac{\log^ 2 n}{n^{3/2}} }   \exp\Big(- (1/2+o(1))  8\pi^2 n^3x_2^2/6   \Big)dx_2.
\end{align*}

After changing variables $y_1= \sqrt{8n\pi^2} x_1$ and  $y_2= \sqrt {4\pi^2n^3/6} x_2 $,  and using the Gaussian  identity $ \frac{1}{\sqrt {2\pi}}  \int^{\infty}_{-\infty}  \exp(-y^2/2 ) dy =1$, we have $S \ge \Omega(n^{-2})$, completing the proof. 

\end{proof}

To complete the picture, we prove Claim \ref{l.endpoint}. In fact, more is true: 

\begin{claim}\label{claim:anticoncentration} Let $u_i=(1,i)$ and $v_i=(1,(-1)^{i-1}i)$, $0\le i\le n$. Assume that $\xi_0,\dots,\xi_n$ are iid copies of a random variable of variance one and bounded $(2+\ep)$-moment. Then we have
\begin{itemize}
\item $\sup_{a\in \R^2}\P(\sum_i \xi_i u_i =a)=O(n^{-2})$;
\vskip .1in
\item $\sup_{a\in \R^2}\P(\sum_i \xi_i v_i =a)=O(n^{-2})$.
\end{itemize} 
 
\end{claim}

%Clearly Theorem \ref{theorem:close:Ber} follows from this result, it remains to verify it.

\begin{proof}[Proof of Claim \ref{claim:anticoncentration}] As the proof for the second estimate is similar, it suffices the prove the first one. We'll use \cite{NgV-advances}. Assume otherwise that $\rho:=\sup_{a\in \R^2}\P(\sum_i \xi_i u_i =a)\ge C_1n^{-2}$ for some sufficiently large $C_1$.  Then by \cite[Theorem 2.5]{NgV-advances} most of the $v_i$ belongs to a symmetric GAP $Q$ of rank $r$ and size 
$$|Q|\le O(\rho^{-1}/n^{r/2}).$$

As $r\ge 2$, $|Q|\le O(C_1^{-1}n)$, which is smaller than $n/2$ if $C_1$ is sufficiently large, a contradiction to the fact that $Q$ contains most of the $u_i$.
\end{proof}

%{\color{blue} To anh Van: we wrote a little bit more here for Type I.}

\section{A remark on Theorem \ref{theorem:inverse}}\label{appendix:inverse}
\subsection{Deduction of Theorem \ref{theorem:inverse} from \cite{NgV-advances}} 
% Notice that $\ell_0=\sqrt{n'/\log^2 b} \le k=\sqrt{n'/64\pi^2 m}$ defined there, and so 

Recall that at the end of the proof of \cite[Theorem 2.9]{NgV-advances} (equation (22)), we obtained the following for the dilated set $V_\beta:=\{\beta^{-1}v_1,\dots, \beta^{-1}v_n\}$: for any $n^\ep_0 \le n'\le n$, there exists a subset $V_\beta' \subset V_\beta$ of size at least $n-n'$ such that  
\begin{align}\label{eqn:sizeofkV'}
\mu_{Lebesgue} \left( k(V_\beta'\cup \{0\})+C(0,\frac{1}{256y_0}) \right)  =O(  \rho^{-1} y_0^{-1}\exp(-\frac{m}{4}+2) m^{1/2}).
\end{align}

where $C(0,r)$ is the open disk of radius $r$; and $1< y_0 =O(1)$ and $1\le m=O(\log n)$, here the implied constants are allowed to depend on $\xi,\beta$ and $C'$. 

Notice that $k=\sqrt{n'/64\pi^2m}$, and as $n$ is sufficiently large, $\ell_0 = \sqrt{n'/\log^2 n} \le k$. Thus we also have
\begin{align}\label{eqn:sizeofkV'}
\mu \left( \ell_0(V_\beta'\cup \{0\})+C(0,\frac{1}{256y_0}) \right)=O( \rho^{-1} y_0^{-1}\exp(-\frac{m}{4}+2) m^{1/2}).
\end{align}

Let $D:=1024y_0$. We approximate each vector $v'$ of $V_\beta'$ by a closest number in $\frac{\Z}{D\ell_0}$,
$$|v'-\frac{a}{D\ell_0}| \le \frac{1}{D\ell_0}, \mbox{ with } a\in\Z.$$

Let $A_\beta$ be the collection of all such $a$. It follows from \eqref{eqn:sizeofkV'} that
\begin{align*}
|\ell_0(A_\beta+C_0(0,1))|&= O(\rho^{-1}\ell_0 \exp(-\frac{m}{4}+2)m^{d/2}) =O(\rho^{-1}\ell_0),
\end{align*}
where $C_0(0,r)$ is the discrete cube $\{(z_1,\dots,z_d)\in \Z^d: |z_i|\le r\}$.

Now we apply Theorem \cite[Theorem 1.21]{TVJohn}, it is implied that the set $\ell_0(A_\beta+C_0(0,1))$ belongs to a GAP $Q_0$ of size $O(\rho^{-1}\ell_0)$, where the implied constant depends on  $\xi,\beta$ and $C'$. Next, by iterating \cite[Theorem 3.40]{TVbook} if necessary, one can assume that $Q_0$ is $C$-proper, while the size $|Q_0|$ remains $O(\rho^{-1}\ell_0)$ but the implied constant now also depends on $C$.

In the next step, one applies \cite[Lemma A.2]{NgV-advances} to ''divide'' the GAP $Q_0$, obtaining a $C$-proper GAP $P\subset \Z$ containing $A_\beta+C_0(0,1)$, which has small rank $1\le r=O(1)$, and small size $|P|=O(\rho^{-1}\ell_0 \ell_0^{-r}) = O(\rho^{-1}).$

Set  $Q=\frac{\beta}{\ell_0D} \cdot P$, then the following holds:

\begin{itemize}
\item $Q$ has small rank, $r=O (1)$, and small cardinality $|Q| =O(\rho^{-1}\ell_0^{1-r}) =O(\rho^{-1})$;
\vskip .1in
\item for all but at most $n'$ elements $a$ of $\{a_1,\dots,a_n\}$, there exists $q\in Q$ such that  

$$|q-a|\le \beta/D\ell_0;$$ 
\vskip .1in

\item $Q$ is $C$-proper;
\vskip .1in
\item as $C_0(0,1)\in P$, there exist $|k_1|\le L_1,\dots, |k_d|\le L_d$ such that 

$$\beta/D\ell_0 =\sum_i k_i g_i;$$
\vskip .1in
\item as $P\subset \Z$, all steps $g_i$ of $Q$ are {\it integral multiples} of $\beta/D\ell_0$.
\end{itemize}

\end{document}